\documentclass[letterpaper, 11pt]{article}
\usepackage[shortalphabetic,abbrev]{amsrefs}
\usepackage{xspace,amsfonts,euscript}
\usepackage{amsmath}
\usepackage{amssymb}
\usepackage{amscd}
\usepackage{setspace}

\usepackage{mathrsfs}

\usepackage[matrix,arrow,graph]{xy}

\usepackage{latexsym}
\usepackage{theorem}
\usepackage{epsfig}
\usepackage{palatino} 

\newtheorem{theorem}{Theorem}[section]
\newtheorem{lemma}[theorem]{Lemma}
\newtheorem{proposition}[theorem]{Proposition}
\newtheorem{corollary}[theorem]{Corollary}

{\theorembodyfont{\normalfont}
\theoremstyle{plain}
\newtheorem{definition}[theorem]{Definition}
\newtheorem{example}[theorem]{Example}

\newtheorem{remark}[theorem]{Remark}
}

\numberwithin{equation}{section}

\newcommand{\rank}{\operatorname{rk}}

\newcommand{\rk}{\operatorname{rk}}

\renewcommand{\Im}{\operatorname{Im}}

\newcommand{\id}{\operatorname{id}}

\newcommand{\Ker}{\operatorname{Ker}}
\renewcommand{\dim}{\operatorname{dim}}
\newcommand{\codim}{\operatorname{codim}}

\newcommand{\Sym}{\operatorname{Sym}}
\newcommand{\Hom}{\operatorname{Hom}}

\newcommand{\cone}{\operatorname{cone}}

\newcommand{\C}{{\mathbb{C}}}
\newcommand{\Z}{{\mathbb{Z}}}
\newcommand{\Q}{{\mathbb{Q}}}
\renewcommand{\H}{{\mathbb{H}}}
\newcommand{\R}{{\mathbb{R}}}

\newcommand{\jhat}{\hat\jmath}

\newcommand{\bigmid}{\hs\Big{|}\hs}
\newcommand{\subs}{\subset}

\newcommand{\hookto}{{\hookrightarrow}}

\newcommand{\impl}{\Rightarrow}
\newcommand{\becircled}{\mathaccent "7017}
\newcommand{\hs}{\hspace{3pt}}

\newcommand{\M}{\mathfrak{M}}
\newcommand{\N}{\mathfrak{N}}

\newcommand{\D}{\Delta}

\newcommand{\sig}{\sigma}
\newcommand{\Sig}{\Sigma}

\renewcommand{\t}{\mathfrak{t}}

\newcommand{\IH}{I\! H}
\newcommand{\IC}{\mathop{\mathbf{IC}}\nolimits}

\renewcommand{\cal}{\mathcal}
\newcommand{\cA}{{\cal A}}
\newcommand{\cB}{{\cal B}}
\newcommand{\cC}{{\cal C}}

\newcommand{\cE}{{\cal E}}
\newcommand{\cF}{{\cal F}}
\newcommand{\cH}{{\cal H}}
\newcommand{\cI}{{\cal I}}
\newcommand{\cL}{{\cal L}}
\newcommand{\cM}{{\cal M}}

\newcommand{\cR}{{\cal R}}
\newcommand{\cS}{{\cal S}}

\newcommand{\ol}{\overline}

\newcommand{\sS}{\mathscr S}
\newcommand{\GB}{\Gamma_{\! B}}

\newcommand{\la}{\langle}
\newcommand{\ra}{\rangle}

\newcommand{\wt}{\widetilde}
\newcommand{\wh}{\widehat}
\newcommand{\md}{\mathrm -mod}
\newcommand{\bdy}{{\partial}}
\newcommand{\Span}{\mathop{\mathrm{span}}}

\newcommand{\Hilb}{\mathop{\rm Hilb}}
\newcommand{\Loc}{\mathop{\rm Loc}}
\newcommand{\Loch}{\mathop{\wh{\rm Loc}}}
\newcommand{\be}{{e}}
\renewcommand{\udot}{{\scriptscriptstyle \bullet}}
\newcommand{\bigmod}{\big{/}}
\newcommand{\sgn}{\operatorname{sgn}}

\newcommand{\ICxIC}{\IC\otimes\IC}

\renewcommand{\emptyset}{\varnothing}

\newcommand{\ilim}{\mathop{\varprojlim}\limits}

\newcommand{\tV}{{\wt V}}
\newcommand{\tcH}{{\wt\cH}}
\newcommand{\hkv}{V^{\text{hk}}}

\newcommand{\hkhi}{H^{\text{hk}}_i}
\newcommand{\hkh}{H^{\text{hk}}}
\newcommand{\hkhf}{H^{\text{hk}}_{F}}
\newcommand{\circhkhf}{\becircled{H}^{\text{hk}}_{F}}

\newcommand{\surj}{\twoheadrightarrow}
\newcommand{\rc}{\rho^{}_{\C}}
\newcommand{\ulA}{{\underline{A}}}

\newcommand{\qed}{\hfill \mbox{$\Box$}\medskip\newline}
\newenvironment{proof}{\noindent {\bf Proof:}}{\qed \par}

\begin{document}

\spacing{1.1}
\noindent
{\Large \bf The hypertoric intersection cohomology ring}
\bigskip\\
{\bf Tom Braden}\footnote{Supported in part by NSF grant DMS-0201823.}\;\;\;\texttt{braden@math.umass.edu
}\\
Department of Mathematics and Statistics, University of Massachusetts, Amherst, MA 01003\\
{\bf Nicholas Proudfoot}\footnote{Supported in part by an NSF
Postdoctoral Research Fellowship and NSF grant DMS-0738335.}\;\;\;\texttt{njp@uoregon.edu}\\
Department of Mathematics, University of Oregon, Eugene, OR 97403
\bigskip
{\small
\begin{quote}
\noindent {\em Abstract.} 
We present a functorial computation of the equivariant intersection cohomology of
a hypertoric variety, and endow it with a natural ring structure.  When the
hyperplane arrangement associated with the hypertoric variety is unimodular,
we show that this ring structure is induced by a ring structure on 
the equivariant intersection cohomology sheaf in the equivariant derived category.
The computation is given in terms of a localization functor which takes
equivariant sheaves on a 
sufficiently nice stratified space to sheaves on a poset.
\end{quote}
}

A hypertoric variety 
is a symplectic algebraic variety, equipped with a torus action, whose structure is
determined by the geometry and combinatorics of a rational hyperplane
arrangement 
in much the same way that a toric variety is determined
by a rational convex polyhedron (or, more generally, a rational fan).  
Since hypertoric varieties were first introduced by Bielawski and Dancer \cite{BD}, 
many of their algebraic invariants have been computed 
in terms of the associated arrangements.  In particular, combinatorial formulas
have been given for the ordinary and equivariant cohomology rings 
of a smooth hypertoric variety \cites{HS,Ko,HP}, 
and for the intersection cohomology Betti numbers of a singular
affine hypertoric variety \cite{PW}.

In this paper, we refine the results of \cite{PW} to give a
combinatorial computation of the equivariant intersection  
cohomology groups of a hypertoric variety $\M_\cH$ associated to an arbitrary
rational hyperplane arrangement $\cH$ (Theorem \ref{main theorem 1}). 
We use this to prove a conjecture of \cite[6.4]{PW}, which states
that the intersection cohomology of a hypertoric variety has a natural ring structure
(Corollary \ref{it's a ring}).  
In the special case where $\cH$ is central and unimodular
(which is equivalent to saying that $\M_\cH$ is affine and 
has a hypertoric resolution of singularities), we show
that this ring structure exists on the deepest possible level, namely on
the equivariant IC sheaf in the equivariant derived category (Theorem \ref{IC is ring object}).
The unit element in this ring structure is given by the
natural map from the constant equivariant sheaf, which 
implies that our ring structure ``behaves like a cup product".  In particular, for example,
it implies that the restriction map from the equivariant intersection cohomology of $\M_\cH$ 
to the equivariant cohomology of the generic stratum will be a ring homomorphism.

We prove these results using a general notion, which we develop in 
Section \ref{linear posets section}, of localization from $T$-equivariant 
constructible sheaves on an equivariantly stratified $T$-space to sheaves on a poset 
whose elements index the strata, equipped with a linear structure that 
keeps track of the stabilizer on the associated stratum.
In this framework, we identify the total equivariant intersection cohomology group
$\IH_T^\udot(\M_\cH)$ with the space of
sections of a sheaf $\cL$, called a minimal extension sheaf, on the poset $L_\cH$ of flats
of $\cH$.  The flats index the strata of a natural 
stratification of $\M_\cH$, and we show that the stalk of $\cL$ at any flat 
is canonically isomorphic to the
``local" equivariant intersection cohomology of $\M_\cH$ at any $T$-orbit in
the corresponding stratum.  Furthermore, the restriction maps in the sheaf 
coincide with the natural maps between local intersection cohomology groups
of comparable strata (in fact, this last property is enough to 
characterize the isomorphism between the space of global sections of $\cL$ and
$\IH^\udot_T(\M)$ up to 
multiplication by a scalar). 
This formalism of localization functors
provides a unified setting in which to understand our work along with a number of 
other theories involving localization of equivariant cohomology or equivariant sheaves, including
\cites{BBFK1, BBFK2, BM, BreLu, GKM}.

The concept of a minimal extension sheaf was originally 
introduced for sheaves on fans in \cites{BBFK1, BreLu}, where the
sections of such sheaves give the equivariant intersection 
cohomology groups of toric varieties.  We generalize this notion to arbitrary linear posets,
focusing on the case of $L_\cH$.
This perspective turns out to be useful for understanding a number of important
aspects of the topology of $\M_\cH$.
In addition to our main theorem identifying $\IH_T^\udot(\M_\cH)$ with the space of sections of $\cL$, we show how
to use $\cL$ to compute the intersection cohomology Morse groups of $\M_\cH$ in the sense of
stratified Morse theory for a generic projection of the moment map on our hypertoric variety.
We also give a combinatorial version of the Beilinson-Bernstein-Deligne decomposition theorem
for the canonical orbifold resolution of $\M_\cH$.
Both of these computations are important for describing a duality relating
perverse sheaves on hypertoric varieties defined by 
Gale dual arrangements, which will appear in the forthcoming paper \cites{BLPW}. 

Part of our motivation for studying hypertoric varieties is that they share many geometric
properties with other symplectic algebraic varieties that play prominent roles in representation
theory and physics, such as quiver varieties,
moduli spaces of Higgs bundles on a curve, and Hilbert schemes of points on symplectic
surfaces.  In light of Theorem \ref{IC is ring object}, we are led to ask whether there is
a broader class of symplectic varieties with natural ring structures on their IC sheaves.  
In \cite[3.4.4]{Pr}, the second author conjectures that many symplectic quotients of vector spaces, 
including all quiver varieties, admit such structures, and he uses the results of this paper
to give an explicit conjectural presentation for the intersection cohomology 
rings that would arise in this way.

Finally, if $\cH$ is an arrangement over a field other than $\Q$ 
and cannot be defined over $\Q$,
then there is no associated hypertoric variety, but our theory of minimal
extension sheaves still makes sense.
For this reason the global sections of the minimal extension sheaf
$\cL$ can be called the intersection cohomology of the 
arrangement, regardless of whether there is a hypertoric interpretation.  
This is analogous to what happens for toric varieties,
where Karu's theorem \cite{Ka} implies that the intersection 
cohomology of a non-rational fan, defined by means of a minimal
extension sheaf, satisfies all of the expected properties 
despite the lack of an associated toric variety.  

\bigskip
\noindent{\em Acknowledgments}.  The first author would like to thank the 
hospitality of the Institute for Advanced Studies at the Hebrew University,
Jerusalem, where some of these results were worked out.  

\begin{section}{Linear posets and localization}\label{linear posets section}
In this first section, we construct a localization functor from the equivariant derived category
of constructible sheaves on a sufficiently nice stratified space to the category of modules
over the structure sheaf of a linear poset.  We also give some basic definitions, results,
and examples pertaining to minimal extension sheaves on a linear poset.

\subsection{Sheaves on posets}\label{sheaves on posets}
Our 
results will be expressed in the language of 
sheaves on finite posets, which we now review.  Suppose that
$(P, \le)$ is a finite poset.  We put a topology on it by 
declaring $U\subset P$ to be open if for every $x \in U$ 
and $y \le x$, we also have $y \in U$.  Then for any
$x\in P$, the set $U_x = \{y\in P \mid y \le x\}$
is the smallest open set containing $x$.
We can also think of $P$ as a small category with the properties that
each Hom set has at most one element, and no two distinct objects are isomorphic; here 
$x\le y$ if and only if there exists a morphism $y \to x$.
\begin{proposition}\label{stalks are enough}
Let $\cC$ be an abelian category.
The category of sheaves on $P$ of objects in $\cC$
is equivalent to the category of functors $P \to \cC$.  In other
words, a sheaf $\cS$ is given by an object $\cS(x)$ for each 
$x\in P$, together with restriction maps $r_{xy}\colon \cS(y) \to \cS(x)$
for every $x,y \in X$ with $x \le y$, satisfying $r_{xy}r_{yz} = r_{xz}$
whenever $x \le y \le z$.
\end{proposition}

\begin{proof}
Given a sheaf $\cS$, we obtain the 
data $\{\cS(x), r_{xy}\}$ by setting $\cS(x) = \cS(U_x)$,
the sections on the open set $U_x$, and letting $r_{xy}$ be the 
restriction map.  In the other direction, given the objects
$\cS(x)$ and maps $r_{xy}$, the sections $\cS(U)$ of the sheaf
$\cS$ on an open set $U$ is defined to be the projective limit
$\ilim_{x\in U} \cS(x)$.
\end{proof}

We will use the following shorthand: for $x\in P$, we let $\bdy x$
denote the ``punctured neighborhood'' 
\[\bdy x := U_x \smallsetminus \{x\} = \{y \in P \mid y < x\}.\]
If $\cS$ is a sheaf on $P$, we let 
$\bdy_x\colon\cS(x) \to \cS(\bdy x)$ denote the 
restriction map from $U_x$ to $\bdy x$, and we put $\cS(x,\bdy x) = \ker \bdy_x$.

Fix a field $k$ (later we will take $k = \R$). 
We define a {\bf linear poset} to be a pair $(P,V)$ of a finite poset $P$
together with a sheaf $V$ of finite-dimensional $k$-vector spaces on it.  
In other words $V$ is a
functor $P \to \operatorname{\mathbf{Vect}}_k$.  
In all of our examples the restriction maps $r_{xy}$ will all be 
surjective. 
Given a linear poset $(P,V)$, its {\bf structure sheaf} $\cA = \cA_P$
is the sheaf of graded polynomial rings $\Sym V$, i.e.\ the sheaf
whose stalk $\cA(x)$ at $x\in P$ is equal to $\Sym V(x)$.  We use the grading
where elements of $V(x)$ have degree two.  Let $\cA\md$ denote the 
category of finitely generated graded $\cA$-modules; we will refer
to them simply as $\cA$-modules. 

Our examples of linear posets will come from spaces with a 
torus action, in the following way.  Let $\M$ be a Hausdorff
topological space on which a compact abelian Lie group $T$ acts 
(in the examples we study $T$ will be a torus, but these 
initial definitions make sense even when $T$ is disconnected).
We define a  \textbf{{\em T}-decomposition} of $\M$ indexed by 
a poset $P$ to be a collection
$\sS = \{S_x\}_{x \in P}$ of locally closed $T$-invariant 
subspaces partitioning $\M$ so that
\begin{itemize}
\item for every $x, y \in P$,\, $S_y \cap \overline{S_x} \ne 
\emptyset \Longleftrightarrow S_y \subset \overline{S_x} \Longleftrightarrow x \le y$, 
and\footnote{Although the opposite convention on the partial order might seem more geometrically
natural, this choice will agree with conventions regarding
fans and hyperplane arrangements.}
\item for every $x\in P$, there is a Lie subgroup $T_x \subset T$ (possibly
not connected) so that the stabilizer of any point of $S_x$ is $T_x$.  
\end{itemize}

Given a $T$-decomposition of $\M$ indexed by $P$, we put a linear structure 
on $P$ by letting $V(x) = (\t_x)^*$, the dual of the Lie algebra of $T_x$.
The definition of the poset structure ensures that 
$T_x \subset T_y$ if $x \le y$, and we let the restriction
$r_{xy}\colon V(y)\to V(x)$ be the map dual to the inclusion
$\t_x \subset \t_y$.

The geometric meaning of the resulting structure sheaf $\cA = \cA_P$
comes from equivariant cohomology.  For any $x\in P$
and any point $p\in S_x$, there is a canonical identification
\[H_T^\udot(Tp) = H^\udot_{T_x}(pt) = \Sym V(x) = \cA(x)\]
of graded rings (all cohomology groups in this paper will be taken 
with coefficients in $\R$).  If $x \le y$, then the map 
$r_{xy}\colon \cA(y) \to \cA(x)$ is the pullback by any $T$-equivariant
projection $Tp_x \to Tp_y$.  

\subsection{The localization functor}
\label{localization functor}

For a large class of $T$-spaces $\M$ endowed with a $T$-decomposition $\sS$,
we can use modules over the structure sheaf $\cA$ to study $T$-equivariant
sheaves on $\M$.  Let $D^b_T(\M)$ denote the triangulated 
category of $T$-equivariant sheaves 
defined by Bernstein and Lunts \cite{BerLu}, and let
$D^b_{T,\sS}(\M)$ be the full subcategory 
of ``$\sS$-constructible'' objects: objects whose cohomology 
sheaves have finite-dimensional stalks, vanish outside of
a finite range of degrees, and are locally constant on every 
$S_x$, $x\in P$.
Let $\mu\colon \M \to \M/T$ denote the quotient map.  
We will construct, under suitable topological hypotheses on $\sS$, 
a functor \[\Loc\colon D^b_{T,\sS}(\M) \to \cA\md\]
so that 
\begin{itemize}
\item for any $x \in P$ and any $p\in S_x$, there is 
a natural isomorphism 
\begin{equation} \label{localized stalks} H^\udot_T(Tp; B)\simeq (\Loc B)(x) 
\end{equation}
of 
$\cA(x)= H^\udot_T(Tp)$-modules\footnote{
Here and later in the paper we use the shorthand notation
$H^\udot_T(X; B)$ to denote the equivariant cohomology
of the pullback $j^*B$, where $j\colon X\to \M$ is the 
inclusion of a $T$-invariant subset $X$ of $\M$.}, and
\item there is a natural homomorphism
\begin{equation}\label{localization homomorphism}
\GB:H^\udot_T(\M; B) \to (\Loc B)(P)
\end{equation}
which is compatible under restriction with the isomorphism
\eqref{localized stalks}.
\end{itemize}

The first condition that we need the pair $(\M,\sS)$
to satisfy is the following, which is an equivariant 
generalization of the notion of topological 
stratification as used by Goresky 
and MacPherson \cite[\S1.1]{GM}.  

\begin{definition} \label{T-stratification}
Let $T$ be a compact Abelian Lie group, let $\M$ be a $T$-space, and let
$\sS = \{S_x\}_{x\in P}$ be a $T$-decomposition of $\M$ indexed by 
a poset $P$.  Then we define the statement that 
$\sS$ is a \textbf{$T$-stratification} inductively on
$|P|$ to mean that for every $x \in P$, and every point $p\in S_x$,
there exists 
\begin{itemize}
 \item an open neighborhood $U \subset \M$ of the orbit
$T \cdot p$,
\item a $T_x$-space $L$, with a $T_x$-stratification
$\{S^x_y \mid y < x\}$, and
\item a $T$-equivariant
homeomorphism 
\[\phi\colon U \cong T \times^{}_{T_x} \cone(L) \times D,\]
where $D$ is the unit disk in a Euclidean space, and $\cone(L)$ 
is the open topological cone $(L \times [0,1))/(L \times \{0\})$,
with the induced $T_x$-action,
\end{itemize}
such that $\phi$ is compatible with the 
induced decompositions on both sides: 
$$\phi(U\cap S_y) = T \times^{}_{T_x} ( S^x_y\times (0,1)) \times D$$
for all $y > x$, and $\phi(U \cap S_x) = T \times^{}_{T_x} \{v\} \times D$, 
where $v \in L$ is the apex of the cone.  When $x$ is a minimal element
of $P$, we take $L = \emptyset$, $\cone(L) = \{v\}$ (this takes care of the
base case $|P| = 1$).  
The space $\cone(L)$ is called the \textbf{normal slice} to 
$S_x$ at the point $p$, $L$ itself is called the \textbf{link} of
$S_x$, and $U$ is called a \textbf{stratified tubular neighborhood}
of $S_x$.
\end{definition}

The following basic properties of this definition are easy to check.
\begin{lemma}\label{stratification remarks}
If $\sS$ is a $T$-stratification of $\M$, then  
\begin{itemize}
 \item for any inclusion $j\colon \N \to \M$ of a locally closed union of some of 
the strata $S_x$, the functors $j^*$, $j_*$, $j^!$, $j_!$ preserve $\sS$-constructibility,
\item for every $x\in P$, $S_x$ and $S_x/T$ are both
topological manifolds, and $S_x \to S_x/T$ is a principal $T/T_x$-fiber bundle, and
\item for any Lie subgroup $T'\subset T$, the decomposition 
$\{S_x/T'\}$ defines a $T/T'$-stratification
of $\M/T'$.
\end{itemize}
\end{lemma}

Given a $T$-space $\M$ with a $T$-decomposition $\sS$, in order
to define our localization functor we will need to assume that
\begin{enumerate}
\item[(A)] $\sS$ is a $T$-stratification, 
\item[(B)] for each $x\in P$, the quotient $S_x/T$ is simply connected, and
\item[(C)] for every $x, y\in P$, $y < x$, the space $S^x_y$ 
given by Definition \ref{T-stratification} is connected. 
\end{enumerate}

Condition (B) ensures that, for $B\in D^b_{T,\sS}(\M)$, 
the spaces $H^\udot_T(Tp; B)$
are all canonically isomorphic as the point $p$ varies in
$S_x$.  Then, intuitively, we would like to define the
map $(\Loc B)(y) \to (\Loc B)(x)$ for $x < y$ by letting
a point $p\in S_x$ move to the boundary and degenerate to
a point in $S_y \subset \ol{S_x}$.  Condition (C) ensures
that the resulting map does not depend on the path $p$ takes.

More formally, we 
define the functor $\Loc$ by pushing forward
sheaves from $\M$ to $P$ in two steps, first by 
the quotient map $\mu\colon\M \to \M/T$, 
and then by the natural map $\pi\colon \M/T \to P$
which sends any point in $S_x/T$ to $x$.
The first step is a pushforward by the functor 
$Q_{\mu *}\colon D^b_T(\M) \to D^+(\M/T)$ which was 
defined by Bernstein and Lunts \cite[6.9]{BerLu}. 
We can describe this functor more concretely as follows. 
Let $ET$ be an acyclic free $T$-space.  The category 
$D^b_T(\M)$ is equivalent to the full 
subcategory of $D^b(\M \times_T ET)$ consisting of objects
whose pullback to $\M \times ET$ is isomorphic to the
pullback of an object of $D^b(\M)$ along the projection
$\M \times ET \to \M$.  Then $Q_{\mu *}$ is the 
(derived) pushforward $R\bar\mu_*$ along the quotient map 
$\bar\mu\colon\M \times_T ET \to \M/T$.  Note the equivariant
cohomology $H^\udot_T(\M;B)$ of an object $B$ of $D^b_T(\M)$ is 
just the ordinary hypercohomology of the corresponding object of $D^b(\M \times_T ET)$.

We then define a graded sheaf $B^\mu$ to be the direct sum of the 
cohomology sheaves of $Q_{\mu *}B$.  The following result 
shows that this sheaf has the stalks that we want.

\begin{lemma} \label{stalk lemma}
For any point $p\in \M$ and $B \in D^b_{T,\sS}(\M)$, there is a natural 
isomorphism\footnote{Note that on the left-hand side $Tp$ is regarded as a subset of $\M$,
while on the right-hand side it is regarded as a point of the quotient $\M/T$.}
\begin{equation}\label{bmustalk}
H_T^\udot(Tp; B)\cong (B^\mu)_{Tp}.\end{equation}
\end{lemma}

\begin{proof} Base change gives a natural map
$H_T^\udot(Tp; B)\to(B^\mu)_{Tp} $.  Suppose that $p\in S_x$.
We have $(B^\mu)_{Tp} = \varinjlim H^\udot(U\times_T ET; B)$, where the limit is 
over $T$-invariant open sets $U$ containing $Tp$.  
If we let $U$ be a stratified tubular neighborhood of $S_x$, as provided by 
Definition \ref{T-stratification}, then we get a saturated sequence $\{U_n\}$ 
of open neighborhoods of $Tp$ by taking 
$\phi(U_n)$ to be the image of $T \times (L \times [0,1/n)) \times (1/n)D$ 
in $T \times_{T_x} \cone(L) \times D$.  The group $H^\udot(U_n \times_T ET; B)$
is independent of $n$, and so we see that 
\[(B^\mu)_{Tp} = H^\udot(U \times_T ET; B) = H^\udot_T(U; B).\]

Note that $D^b_{T,\sS}(\M)$
is generated as a triangulated category by the objects
$j_{y!}\R_{S_y,T}$, $y \in P$, where $j_y\colon S_y \hookrightarrow \M$
is the inclusion.
Since both sides of Equation \eqref{bmustalk}
are cohomological functors in $B$, it is enough to 
deal with the case $B = j_{y!}\R_{S_y,T}$.
The map $U \times_T ET \to BT$ is a fiber bundle 
with fiber $U$; using this and the structure of 
the induced stratification on $U$ it follows that 
when $x \ne y$, the derived pushforward of 
$B|_{U \times_T ET}$ to $BT$ is zero, so 
$H^\udot(U \times_T ET; B) = 0$.  On the other 
hand, if $x = y$, then 
\[H^\udot_T(U; B) = H^\udot_T(U \cap S_x) = H^\udot_T(T \times_{T_x} \{v\} \times D) = 
H^\udot_{T_x}(\{v\} \times D) = \cA(x).
\]
 The result follows.
\end{proof}

For the second step, we define the localization of 
$B$ to be \[\Loc B = \pi_*B^\mu.\] 
Note that here $\pi_*$ is the ordinary
pushforward of sheaves, {\em not\/} the derived pushforward.  
Because of this, and because $B^\mu$ involves taking
cohomology sheaves, this functor loses a lot of information.  
However, we shall see that it is well-behaved when applied to 
intersection cohomology sheaves, and it can be easily computed.
Lemma \ref{stalk lemma} and the following lemma imply the 
existence of the isomorphism \eqref{localized stalks}.

\begin{lemma}\label{Localized stalks}
Let $E$ be any sheaf on $\M/T$ which is locally constant
(hence constant) on $S_x/T$ for every $x\in P$.  Then the restriction 
\[(\pi_*E)(x) \to \Gamma({S}_x/T, E)\]
is an isomorphism for all $x\in P$.
\end{lemma}

\begin{proof}  
The stalk $(\pi_*E)(x)$ is the same as the space of sections
$\Gamma(U_x; \pi_*E) = \Gamma(\pi^{-1}(U_x); E)$, where $U_x$ is
the smallest open set containing $x$.
In other words, we need to show
that \[\Gamma(\pi^{-1}(U_x), E) \to \Gamma({S}_x/T, E)\]
is an isomorphism.  
Our assumptions (B) and (C) on our stratification imply
that this holds when $E = j_{y*}L$, where $L$ is 
a local system on ${S}_y/T$, $y$ is any element of $P$, 
and $j_y\colon {S}_y/T \to \M/T$ is the inclusion.  
For general $E$, 
apply the exact sequence of sheaves
\[0 \to E \to \bigoplus_{y \in P} j_{y*}j_y^*E \to 
\bigoplus_{z \le w \in P} j_{w*}j_w^*j_{z*}j_z^*E.\]
We have shown that our map is an isomorphism for the 
second and third terms, so it is for $E$, as well.
\end{proof}

At this point $\Loc B$ is just a sheaf of graded vector
spaces.  To make it into 
an $\cA$-module, first note that there
is an identification of the ring $A = H^\udot_T(pt)$
with the graded endomorphisms of the constant 
equivariant sheaf $\R_{pt,T}$ in $D^b_T(pt)$.
Pulling back $\R_{pt,T}$ to 
$\M$ and tensoring gives an action of $A$ on 
any object $B \in D^b_{T,\sS}(\M)$, and hence on the
localized sheaf $\Loc B$.  
The resulting action of of $A$ on the stalk
of $(\Loc B)(x)$ agrees with the action on 
the equivariant hypercohomology $H^\udot_T(Tp)$, 
$p\in S_x$, under the identification 
$(\Loc B)(x) \simeq H^\udot_T(Tp; B)$.  
In particular, it factors 
through the quotient $\cA(x)$, making $\Loc B$ into 
an $\cA$-module, as required. 


Finally, to define the map $\GB: H^\udot_T(\M; B)\to(\Loc B)(P)$,
note that the hypercohomology of $Q_{\mu*}B$ is just the 
global equivariant cohomology $H^\udot_T(\M; B)$.  This 
then has a natural graded homomorphism to the global sections of
the sheaf $B^\mu$, which is isomorphic to the global sections of 
$\Loc B$.  It is easy to check that this is a map of $A$-modules.

\begin{remark}
In certain exceptional cases $\GB$ will be an isomorphism;
one of our main results is that this happens when $\M$ is a hypertoric variety
and $B$ is the equivariant intersection cohomology sheaf on $\M$.
More examples in which $\GB$ is an isomorphism are given in the following section.
\end{remark}

\subsection{Relations with other theories}\label{linear poset examples}
The localization functor we have just defined generalizes constructions 
that have been used to study equivariant cohomology and equivariant 
sheaves in a number of different settings.  In order to put our results
in context, we point out some of these connections in this section. 

One case that has been extensively studied is when
$\M$ is a toric variety defined by a rational fan $\Sigma$, 
$\sS$ is the stratification
by orbits of the complex torus $T_\C$, 
and $T$ is the maximal compact subgroup of $T_\C$.  The poset $P$ is the 
fan $\Sigma$ itself, ordered by inclusion of cones.
The linear structure is given by $V(\sigma) = \Span(\sigma)^*$ 
for any cone $\sigma\in \Sigma$.  The resulting structure sheaf
$\cA$ is also known as the sheaf of ``conewise polynomial functions''
on $\Sigma$: its sections on an open set (subfan) 
$\Sigma'\subset \Sigma$ is the ring of real-valued functions on 
the support of $\Sigma'$ which restrict to polynomial functions 
on each cone.

For any space satisfying the assumptions of \S\ref{localization functor}, 
the structure sheaf $\cA$ will be isomorphic to the
localization of the equivariant constant sheaf $\R_T$ on $\M$.
When $\M$ is a toric variety,  the ring $\cA(\Sigma)$
of global sections is naturally isomorphic to the equivariant Chow cohomology ring
of $\M$ \cite[Thm. 1]{Pa}.  
If  the fan $\Sigma$ is simplicial, then $\M$ is rationally smooth
(that is, it has at worst orbifold singularities), 
and the map 
$\Gamma_{\!\R_T}$ of \eqref{localization homomorphism} 
will be an isomorphism.
For arbitrary rational fans, if we take $B$ to be the equivariant intersection cohomology sheaf $\IC_T(\M)$,
the map $\GB:\IH^\udot_T(\M)\to \Loc(B)(\Sigma)$ is an isomorphism \cite[Theorem 2.2]{BBFK1}.

\begin{remark}
Because toric varieties
have such simple geometry, the authors of \cites{BBFK1} were able to use a simpler 
(but equivalent) construction in place of our localization functor, defining 
a presheaf on a fan by taking equivariant intersection cohomology 
on open unions of $T_\C$-orbits.   The statement that $\GB$ is an isomorphism
is equivalent to saying that this presheaf is a sheaf.
These papers also gave an algorithm for computing the sheaf 
$\Loc(\IC_T(\M))$ by showing that it is a minimal extension sheaf, a 
concept which we discuss in the next section.
\end{remark}

Our localization functor can also be used to express aspects
of the theory of moment graphs.  Suppose $X$ is
a proper normal algebraic variety $X$ over $\C$ endowed
with an algebraic action of a torus $T_\C$, and let $\M \subset X$ be the
subvariety which is the union of all the zero and one-dimensional orbits
of $X$.  If $\dim_\C \M = 1$, it will consist of a collection of projective lines
joined at $T_\C$-fixed points.  The resulting linear poset can be viewed
as a graph with directions assigned to an edge, which has sometimes
been called the {\bf moment graph} of $X$. Goresky, Kottwitz and MacPherson showed that if $X$ is
{\bf equivariantly formal} \cite[\S 1.2]{GKM} (this will hold for instance
if $X$ is smooth and projective, or if it admits a paving by affines), then
a theorem of Chang and Skjelbred \cite[Lemma 2.3]{CS} implies that $H^\udot_T(X)$
is isomorphic as a ring to $\cA(P)$, via the
composition of the restriction $H^\udot_T(X)\to H^\udot_T(\M)$ with
$\Gamma_{\R_{\M,T}}\colon H^\udot_T(\M) \to \cA(P)$.
Guillemin and Zara \cites{GZ1,GZ2,GZ3} have studied many
aspects of the geometry of moment graphs for smooth varieties $X$
(using a stronger definition reflecting more of the geometry of $X$)
and their cohomology.

If $X$ is singular, then it may fail to be equivariantly formal, but at least when
it is projective $X$ will be equivariantly formal for intersection cohomology,
which implies that $\IH^\udot_T(X)$ is isomorphic as an $H_T(pt)$-module
to the global sections of the sheaf $\Loc(\IC_{X,T}|_{\M})$.  When $X$ is
a Schubert variety in a flag variety, the first author and MacPherson
\cite[Theorem 1.5]{BM} showed that this sheaf can be computed using a universal property
similar to the one satisfied by minimal extension sheaves.

\subsection{Pure $\cA$-modules and minimal extension sheaves} \label{pure sheaves}
The concept of pure sheaves and minimal extension sheaves 
will be central throughout this paper.  Such objects have previously been defined and studied only in the special
case of a fan, described in Section \ref{linear poset examples}.
Nonetheless, the definitions and the basic results concerning them generalize without difficulty,
and we refer the reader to \cites{BBFK2,BreLu} 
for proofs of Proposition \ref{decomposition theorem} and Lemma \ref{rigidity criterion}.

\begin{definition}\label{pure}
For an arbitrary linear poset $(P,V)$, we call an $\cA$-module $\cM$
{\bf pure} if 
\begin{itemize}
\item it is pointwise free: for every $x \in P$, $\cM(x)$
is a free $\cA(x)$-module, and 
\item it is flabby: for any $U\subset P$
open, the restriction $\cM(P) \to \cM(U)$ is surjective.  
\end{itemize}\end{definition}

\begin{definition}
If $M$ is a graded module over a polynomial ring $\Sym V$,  
we define $$\ol{M} = M/VM,$$
and for a homomorphism $\phi\colon M_1 \to M_2$ of graded modules, we let
$\ol\phi\colon \ol{M_1} \to \ol{M_2}$ denote the induced homomorphism 
of graded vector spaces.  
\end{definition}

The following proposition gives a classification of pure
sheaves on a linear poset.  See \cite{BBFK2}*{Proposition 1.3 and Theorem 2.3}, 
\cite[Theorem 5.3]{BreLu} for proofs.

\begin{proposition} \label{decomposition theorem} For each $x\in P$, there 
is an indecomposable pure
$\cA_P$-module $\cL_x$, unique up to isomorphism,
with the property that $\cL_x(y) = 0$ if $x \not\le y$
and $\cL_x(x) \cong \cA(x)$.  For any pure sheaf $\cM$, there is 
an isomorphism
\begin{equation}\label{decomposition}
\cM\cong \bigoplus_{x\in P} \ker \ol{\bdy_x}\otimes_k \cL_x.
\end{equation}
\end{proposition}
The sheaf $\cL_x$ is known as a {\bf minimal extension sheaf}.  

\begin{remark}
The isomorphisms in Proposition \ref{decomposition theorem} are not in general canonical.
\end{remark}

\begin{remark}
Note that a sheaf $\cM$ is flabby if and only if
the restriction $\bdy_y\colon\cM(y) \to \cM(\bdy y)$ is surjective
for all $y\in P$.  The fact that $\cL_x$ is both flabby and indecomposable means that
\[\ol{\bdy_y}\colon \ol{\cL_x(y)} \to \ol{\cL_x(\partial y)}\]
is surjective when $y=x$ and an isomorphism for all other $y\in P$.
\end{remark}

We next give a pair of examples that illustrate the phenomena that we will encounter
in Section \ref{Lattice of flats}.
Let $V$ be a two-dimensional real vector space, and let $\ell_1,\ell_2$, and $\ell_3$
be three distinct lines in $V$.

\begin{example}\label{first}
Consider the poset\\

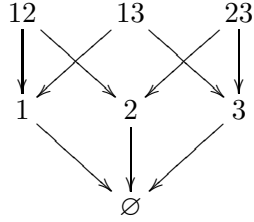
\begin{figure}[!h]
\centerline{
\xymatrix{
12\ar[d]\ar[dr] & 13\ar[dl]\ar[dr] & 23\ar[dl]\ar[d]\\
1\ar[dr] & 2\ar[d] & 3\ar[dl]\\
& \emptyset
}}
\end{figure}

\noindent
with linear structure given by
$V(\emptyset) = 0$, $V(i) = V/\ell_i$, and $V(ij) = V$, with the obvious restriction maps.
It is easy to check that the structure sheaf $\cA$ is flabby, and is therefore a minimal extension
sheaf $\cL_\emptyset$ for the minimal element $\emptyset$.  See Lemma \ref{Q-smooth MES}
for a generalization of this example.
\end{example}

\begin{example}\label{second}
On the other hand,
consider the poset

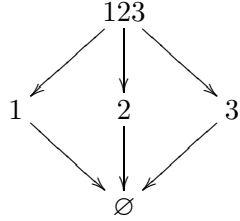
\[
\xymatrix{
& 123\ar[dr]\ar[d]\ar[dl]\\
1\ar[dr] & 2\ar[d] & 3\ar[dl]\\
& \emptyset
}
\]

\noindent
obtained by collapsing the three maximal elements of the poset in Example \ref{first}
into a single element $123$,
and putting $V(123) = V$.
Now the structure sheaf $\cA$ is no longer flabby:
the restriction map $$\cA(123)\to\cA(\partial 123)$$ is not surjective.  Thus $\cL_\emptyset(123)$
must have some ``extra stuff'' to correct this problem.
Indeed, we have
$$\cL_\emptyset(\emptyset) = \cA(\emptyset) = \Sym V(\emptyset) = \R,$$
$$\cL_\emptyset(i) = \cA(i) = \Sym V(i)\hspace{10pt}\text{for $i=1,2,3$,}$$
$$\text{and}\hspace{10pt}\cL_\emptyset(123) = \Sym V\oplus\Sym V[-2].$$
To define the restriction maps, we must specify the image of the generator of $\Sym V[-2]$
in $\cL_\emptyset(\partial 123)_2 \cong V/\ell_1\oplus V/\ell_2\oplus V/\ell_3$.
The only requirement is that it should not come from a single element of $V$.
In fact, any two choices satisfying this condition 
will define the same sheaf up to a unique isomorphism; the fact that there is no
natural choice reflects the fact that there is no natural basis of the free $\Sym V$-module
$\cL_\emptyset(123)$.
\end{example}

\begin{remark} \label{decomp theorem remark}
When the linear poset comes from a rational fan as described in 
\S\ref{linear poset examples},
Proposition \ref{decomposition theorem} is a 
combinatorial version of the decomposition theorem 
of Beilinson, Bernstein, and Deligne \cite{BBD} for toric resolutions of 
singularities (or more precisely its equivariant version proved by 
Bernstein and Lunts \cite[\S5.3]{BerLu}).  
A resolution $\varpi\colon \wt X \to X$ of
toric varieties arises from a subdivision of a rational fan $\Sig$ into a 
smooth fan $\wt\Sig$.  The decomposition theorem says that the pushforward
$\varpi_*\R_{\wt X, T}$ of the constant equivariant sheaf splits as a 
direct sum of shifted intersection cohomology sheaves of subvarieties
of $X$.  Because $\varpi$ is $T_\C$-equivariant, these subvarieties
must be closures of $T_\C$-orbits.  

If $O_\sig$ is the orbit corresponding to $\sig\in\Sig$, 
then the localization of $\IC_T(\ol{O_\sig})$
is the minimal extension sheaf $\cL_\sig$. 
On the other hand, pushing forward the structure sheaf
$\cA_{\wt\Sig} = \Loc \R_{\wt X, T}$ to $\Sig$ gives an $\cA_\Sig$-module $\cE$,
which is easily seen to be pure.
It is not hard to show that localization commutes with  
pushforwards, so $\cE = \Loc \varpi_* \R_{\wt X, T}$.  The
decomposition theorem thus implies that $\cE$ splits into a direct 
sum of shifts of the minimal extension sheaves $\cL_\sig$, $\sig\in \Sigma$,
as guaranteed by Proposition \ref{decomposition theorem}.
\end{remark}

We say that an $\cA_P$-module $\cE$ is {\bf rigid}
if the group of (grading-preserving) automorphisms of $\cE$
consists only of multiplications by nonzero scalars.  This notion is important
because it tells us that if $\cE$ is rigid and another $\cA_P$-module $\cF$ is isomorphic to
$\cE$, then that isomorphism is (nearly) canonical.  In our applications
the groups $\cE(P)$ and $\cF(P)$ will have degree zero parts that are naturally identified
with the base field $k$, which will allow us to make our isomorphisms completely
canonical.
We will use the following criterion to establish rigidity in the 
proof of our first main result, 
Theorem \ref{main theorem 1}.  See \cite[Remark 1.8]{BBFK2}
for a proof.

\begin{lemma} \label{rigidity criterion}
A minimal extension sheaf $\cL_x$ is rigid  if and only if
for each $y\in P$ there exists a number $d$ so that
$\cL_x(y)$ is generated in degrees $\le d$ and
$\cL_x(y, \bdy y)$ is generated in degrees $> d$.
\end{lemma}
\end{section}

\begin{section}{Hyperplane arrangements and hypertoric varieties}\label{hyper}
In this section we review a number of constructions related to hyperplane arrangements.
We explain how to associate to an arrangement a linear poset (the lattice of flats) 
and an algebraic variety (the hypertoric variety).  Our main purpose is to state 
Theorem \ref{main theorem 1}, which relates the equivariant intersection cohomology sheaf
on a hypertoric variety to a sheaf on the lattice of flats via the localization 
functor of Section \ref{linear posets section}.

\subsection{Hyperplane arrangements}\label{arrangements}
We briefly describe the notation and main constructions 
for hyperplane arrangements that we will use.
Let $I$ be a finite indexing set, and   
let $V$ be an affine linear subspace
of $\R^I$ which is not contained in any translate of a coordinate subspace
$\R^J$, $J\subsetneq I$.
Then we consider the collection $\cH$ of affine hyperplanes in $V$ 
formed by intersecting $V$ with the coordinate hyperplanes of 
$\R^I$:
$$H_i := V \cap \{x\in \R^I \mid x_i = 0\}.$$
Note that this is technically a multi-set; there is no reason
why the $H_i$ must be distinct, and in fact the phenomenon of repeated
hyperplanes will be forced upon us by one of the constructions that we will define presently.
The hyperplanes are also cooriented, meaning that they come equipped with normal vectors,
namely the restrictions of the coordinate linear forms on $\R^I$.
When we refer to an {\bf arrangement} in $V$, we will always mean a multi-set of 
cooriented affine hyperplanes in $V$ whose normal vectors span $V^*$.

A {\bf flat} of $\cH$ is a subset of $I$ of the form
$\{i \in I \mid x \in H_i\}$
for some $x \in V$.  Given a flat $F\subs I$, we set 
$$H_F := \bigcap_{i\in F} H_i.$$  This gives a bijection 
between the set of flats and the set of all possible
nonempty intersections of the hyperplanes $H_i$.
Let $L_\cH$ denote the poset of all flats of $\cH$ ordered by inclusion:
$E \le F$ whenever $E\subs F$, or equivalently $H_E \supset H_F$.  
It is a ranked poset; the
rank $\rk F$ of a flat $F$ is the codimension of $H_F$ in $V$, and the
rank of $\cH$ is $\dim V$, since a maximal flat $F$ has $H_F = 0$.

If $V$ is a vector subspace of $\R^I$ (that is if it contains the origin), 
the arrangement $\cH$ will be called {\bf central}.  Note that a central 
arrangement has a unique maximal flat, namely $I$ itself, and $H_I = \{0\}$.  
At the other end of the spectrum, if 
$V\subs\R^I$ is generic with respect to translation, 
then $|F| = \codim H_F$ for every flat $F$.
Such an arrangement will be called {\bf simple}.

For any flat $F$ of $\cH$, we may define two auxiliary arrangements as follows.
The {\bf restriction} of $\cH$ at $F$, 
denoted $\cH^F$, is
the arrangement defined by the inclusion of $H_F$ into
$$\R^{I\smallsetminus F}:=\{x\in\R^I\mid x_i=0\,\,\text{for all $i\in F$}\}.$$  
It is the arrangement
in $H_F$ with hyperplanes
$H_i \cap H_F$ for all $i \notin F$.  Its lattice of flats is isomorphic
to the ideal $\{E \in L_\cH \mid E \ge F\}$.

The {\bf localization} of $\cH$ at $F$, denoted $\cH_F$, 
is the arrangement given by the 
inclusion of $\pi_F(V)$ into $$\R^F := \R^I/\R^{I\smallsetminus F},$$ where
$\pi_F\colon \R^I \to \R^F$ is the coordinate projection.
Its hyperplanes are $\pi_F(H_i)$ for all $i\in F$, i.e.\ the images
of the hyperplanes of $\cH$ which contain $H_F$.
The lattice of flats $L_{\cH_F}$ 
is isomorphic to the interval $[\emptyset, F]$
of $L_\cH$.  The localization $\cH_F$ is always central, with 
$F$ as its unique maximal flat; this is because $\pi_F(H_F) = \{0\}$, 
so the origin is contained in $\pi_F(V)$. 

For any flat $F$, let $\la F\ra$ denote the {\bf linearization} of 
$H_F$, by which we mean the linear subspace of $\R^I$ obtained by translating $H_F$ back to the origin.
Let $V_0 := \la\emptyset\ra$ be the linearization of $V$, so that
$\la F \ra = V_0 \cap \R^{I\smallsetminus F}$ for any flat $F$. 
Let $$V(F) := \pi_F(V) = \pi_F(V_0) = V_0/\la F\ra;$$ 
it is the normal space of the inclusion 
of $H_F$ into $V$.  

The poset $L_\cH$ naturally becomes
a linear poset if we associate
to each flat $F$ the vector space $V(F)$, and to each 
pair $E \le F$ the natural quotient map $V(F)\to V(E)$.
We will denote the structure sheaf of this linear poset
by $\cA$ when the arrangement is clear from the context and
$\cA_\cH$ when it isn't.  Note that every stalk $\cA(F)$ of this sheaf
is a quotient of a single polynomial ring $A = \Sym V_0$.

\begin{remark}\label{simple}
The arrangement $\cH$ is simple if and only if every localization $\cH_F$ is a 
normal crossings arrangement, defined by the inclusion of
$\R^F$ into itself.  In other words, simplicity is equivalent to 
the property that the inclusion $V(F) \hookrightarrow \R^F$ is an isomorphism
for each flat $F$.
\end{remark}

For any arrangement $\cH$ defined by $V \subset \R^I$, 
we can find a translation $\tV$ of $V\subs\R^I$ with the property that 
the associated arrangement $\tcH$ is simple.
If we identify $\tV$ with $V$ by an affine transformation of $\R^I$,
the hyperplane $\wt{H}_i$ of $\tcH$ is a translation of the hyperplane $H_i$ of $\cH$,
and simplicity of $\wt{H}$ means 
that these translations are maximally generic with respect to intersections.
We then have a canonical surjection $$\pi:L_{\tcH}\to L_\cH,$$ 
taking each flat of $\tcH$ to the minimal
flat of $\cH$ that contains it, with the property that $$\pi^*\cA_{L_\cH} = \cA_{L_{\tcH}}.$$
We will refer to the arrangement $\tcH$ as a {\bf simplification} of $\cH$.
Observe that a simplification $\tcH$ of $\cH$ induces a simplification $\wt{\cH_F}$
of the localization $\cH_F$ at any flat $F$.

Finally, suppose that $V\subset \R^I$ is defined over the rational
numbers.  We say that the resulting arrangement $\cH$ is
{\bf unimodular} if
for each subset $S\subset I$, the projection of the lattice
$V\cap \Z^I$ to $\Z^S$ has no cotorsion.   Unimodularity depends
only on $V_0$ (rather than its translation $V$), and it is preserved by 
restriction and localization.

\begin{remark}\label{seymour}
Seymour \cite{Se} gives what amounts to a classification of unimodular
arrangements up to repetition and translation of hyperplanes.
More precisely, he shows that every central unimodular arrangement
can be built out of graphic arrangements (subarrangements of the
braid arrangement), their Gale duals, and one exceptional example,
using three elementary gluing constructions.
See \cite[1.2.5, 3.1.1]{Wh} for a detailed statement of this result.
\end{remark}

\begin{remark}
The linear poset structure on $L_\cH$ and the one on the set of
cones in a fan have a common generalization.  Let $W$ be a 
finite dimensional vector space, and take a collection 
$\{u_i \mid i \in S\}$ of nonzero vectors 
in $W$ indexed by a finite set $S$.  Then for any subset $E \subset S$, we
put $V(E) := \Span\{u_i \mid i\in E\}^*$.  If $P \subset 2^S$ is
a collection of subsets of $S$, ordered by inclusion, this 
defines a linear poset structure on $P$ whose restriction maps are the 
natural quotient morphisms.  

If $\Sigma$ is a fan in $W$, the associated linear poset 
is given by taking $S$ to be the collection of $1$-cones, 
$u_\rho$ any nonzero vector in $\rho$, and 
$P$ the collection of all 
sets of the form $\{\rho \in S \mid \rho \subset \sig\}$ where
$\sig \in \Sigma$.  
If $\cH$ is an arrangement in $V$ with indexing set $I$, we get 
our linear poset structure on $L_\cH$ 
by taking $W = V_0^*$, $S = I$, $P = L_\cH$, 
and letting $u_i$ be the normal vector
to the $i$th hyperplane (or in other words the restriction of
the coordinate function $x_i\colon \R^I \to \R$ to $V_0$) 
for any $i\in I$.  

It is then easy to see that Example \ref{second} is the linear
poset associated to a central arrangement of three lines in the plane 
whose normal vectors sum to zero, while Example 
\ref{first} comes from any simplification of this arrangement.  
\end{remark}

\subsection{Matroid and broken circuit complexes}\label{mat}
Here we collect a few definitions and well-known results from algebraic combinatorics
that we will need in Section \ref{Lattice of flats}.
A {\bf simplicial complex} $\D$ on the ground set $I$ is a nonempty 
collection of subsets of $I$, called {\bf faces}, that is closed under inclusion:  
$S'\subs S\in\D\impl S'\in\D$.  Given such a $\D$, its {\bf face ring}
$\R[\D]$ is defined to be the quotient of the polynomial ring $\Sym\R^I = \R[e_i]_{i\in I}$
by the ideal generated by square-free monomials of the form $e_S := \prod_{i\in S}e_i$ for $S\subs I$ {\em not} a face.
To fit with our interpretation of this ring as equivariant cohomology, we place
the generators $e_i$ in degree two.

At several points we will make use of the fact that 
an inclusion $\Delta' \subset \Delta$ of simplicial complexes induces
a canonical homomorphism $\R[\Delta] \to \R[\Delta']$ of face rings
by sending the generator $e_i$ of $\R[\Delta]$ to the corresponding
generator $e'_i$ of $\R[\Delta']$ if $\{i\} \in \Delta'$, and to 
$0$ otherwise.  We will refer to this 
as the restriction homomorphism dual to the inclusion.

The simplicial complex $\D$ is called {\bf Cohen-Macaulay} if there exists a vector subspace $W\subs\R^I$
such that $\R[\D]$ is a finitely generated free module over the polynomial ring $\Sym W$.  
In particular this implies that all of the maximal faces of $\Delta$ must have cardinality $d = \dim W$.
If $\Delta$ is Cohen-Macaulay, 
its {\bf $\mathbf{h}$-polynomial} $h_\Delta(q)$ is defined to the Hilbert
series of $$\R[\Delta]\otimes_{\Sym W}\R$$ with degrees reduced by half.  In other words, the
coefficient of $q^k$ is the number of generators of $\R[\Delta]$ over $\Sym W$ in degree $2k$.
It can be computed by the formula
\[h_\Delta(q) = \sum_{k \ge 0} f_k(\Delta)\, q^k\, (1-q)^{d-k},\]
where $f_k(\Delta)$ is the number of faces of $\Delta$ of cardinality
$k$ (i.e.\ simplices of dimension $k-1$) \cite[\S II.2]{St}.  

There are two classes of simplicial complex that will interest us.
The first is the {\bf matroid complex} associated to an affine linear
subspace $V\subs\R^I$, where a set $S\subs I$ is a face if and only if
the composition $V\hookrightarrow\R^I\surj\R^S$ is surjective.  We
denote this complex $\D_\cH$.  Note that with our conventions, a
central arrangement and its simplification have the same associated
matroid complex: a subset $S\subs I$ is a face if and only if the
normal vectors to the corresponding hyperplanes form an independent
set.  If the arrangement is simple, then faces of $\Delta_\cH$ 
may equivalently be
characterized as sets of hyperplanes with nonempty intersection.
In other words, the posets $\Delta_\cH$ and $L_\cH$ are the same
for a simple arrangement $\cH$.

The second type of simplicial complex that we need is the {\bf
broken circuit complex} of a matroid complex.  Given a matroid complex
$\D$, a {\bf circuit} is a minimal subset of $I$ that is not a face of
$\D$.  Given an ordering $\sigma$ of $I$, a {\bf broken circuit} is a
set obtained by removing the $\sigma$-minimal element of a circuit,
and the broken circuit complex $\D^{bc}_\cH$ is defined to be
subcomplex of $\D$ consisting of those faces that do not contain any
broken circuit.

Both matroid complexes and their broken circuit complexes are
Cohen-Macaulay, with the subspace $V_0\subs\R^I$ serving as an
appropriate $W$ (see, for example, \cite[\S 4]{HS} and \cite[Prop 1]{PS}).  
While the broken circuit complex depends
on the ordering $\sigma$, its $h$-polynomial does not \cite[\S 7.4]{Bj}.  We
will denote the $h$-polynomials of $\D_\cH$ and any of its broken
circuit complexes by $h_\cH(q)$ and $h_\cH^{bc}(q)$, respectively.
The degree of $h_\cH(q)$ is less than or equal to the rank of $\cH$,
while the degree of $h_\cH^{bc}(q)$ is strictly less than the rank of
$\cH$ \cite[\S 7.4]{Bj}.

\subsection{Hypertoric varieties}
\label{Hypertoric basics}
In this section we assume that the affine linear subspace $V\subs\R^I$ is spanned by its rational points,
and we explain how to use it to construct
a complex algebraic variety called a {\bf hypertoric variety}, originally introduced
by Bielawski and Dancer \cite{BD}.  We will mainly be interested in
how hypertoric varieties look topologically, as stratified spaces, using the stratification
introduced in \cite[\S2]{PW}.  

Consider the coordinate torus $T^I = U(1)^I$.  We will think of $T^I$
as the torus whose Lie algebra is {\em dual} to $\R^I$; in other words, we will
identify $T^I$ with the quotient of $(\R^I)^*$ by the standard lattice $(\Z^I)^*$.
Define an action of $T^I$ on the complex symplectic
vector space $\H^I = T^*\C^I = \C^I \times (\C^I)^*$ by
letting $T^I$ act by coordinate multiplication
on $\C^I$, and by the contragredient action on $(\C^I)^*$.
Then $\H^I$ carries a natural hyperk\"ahler structure so that
this action is hyperhamiltonian with 
moment map 
$$\Psi = (\Psi_{\R}, \Psi_{\C}): \H^I \to \R^I\times\C^I \,\cong\, \R^I\otimes_\R\Im\H.$$
The components $\Psi_{\R}$ and $\Psi_{\C}$ of $\Psi$ are given by the formulas
\[\Psi_{\R}(z, w) = \sum_{i\in I} (|z_i|^2 - |w_i|^2)\delta_i\,\,\, \text{ and }\,\,\,
\Psi_{\C}(z,w) = \sum_{i\in I} 2z_iw_i\delta_i,\]
where $\delta_i$ is the $i^\text{th}$ standard basis vector in $\R^I$ or $\C^I$.

The map $\Psi$ is $T^I$-invariant; in fact, the fibers are single
$T^I$-orbits, so $\Psi$ identifies the quotient $\H^I/T^I$ with the target $\R^I\otimes_\R\Im\H$.  
The stabilizer of a point $(z, w) \in \H^I$ is the coordinate subtorus 
$$T_{(z,w)} := \{t\in T^I \mid t_i \ne 1 \implies z_i = w_i = 0\;\text{for all}\;i\in I\}.$$

The hypertoric variety $\M_\cH$ is a hyperk\"ahler quotient of 
$\H^I$ by a subtorus of $T^I$.  More precisely,
let $$T_{V_0} =  {V_0}^\perp\bigmod (V_0^\perp\cap(\Z^I)^*)
\subs (\R^I)^*\bigmod (\Z^I)^*$$ be the 
subtorus of $T^I$ whose Lie algebra is $V^\perp_0$, and let
$$\hkv := V\times V_0^\C\subs\R^I\times\C^I.$$
Then
$$\M_\cH := \Psi^{-1}(\hkv)/T_{V_0}$$ is defined to be the quotient of the preimage of $\hkv$ by the action of $T_{V_0}$.
The hypertoric variety $\M_\cH$ is rationally smooth
if and only if the arrangement $\cH$ is simple, and
it is smooth if and only if $\cH$ is simple and unimodular \cite[3.2 \& 3.3]{BD}.

\begin{remark}
We have just defined hypertoric varieties in the manner most convenient
for our purposes here, but there are also more algebraic ways
to think about these spaces.  First, the hyperk\"ahler
quotient may be thought of as a holomorphic symplectic quotient, in which
the torus $T_{V_0}$ is replaced by its complexification and the real moment map
equation is replaced by a stability condition.  Alternatively, a smooth hypertoric
variety may be defined as a union of the cotangent bundles of the toric
varieties that appear in Remark \ref{core}.  See \cite[\S1.1 \& Rmk. 2.1.6]{Pr}
for more details.
\end{remark}

Consider the affine subspace arrangement of $\hkv$ consisting of the subspaces
$$\hkhi := H_i\times (H_i)_0^\C$$ for all $i\in I$, where $(H_i)_0^\C\subs V_0^\C$ is the complexification
of the linearization of $H_i$.
For all flats $F$ of $\cH$, let $$\hkhf = \bigcap_{i\in F} \hkhi\,\,\, \text{ and } \,\,\,
\circhkhf := \hkhf \smallsetminus \bigcup_{E < F} H^{\text{hk}}_{E}.$$

Let $T = T^I/T_{V_0}$ be the torus with Lie algebra $V_0^*$.  The action of $T^I$ on $\H^I$
induces an action of $T$ on $\M_\cH\,$ with hyperk\"ahler moment map
$\mu\colon \M_\cH \to \hkv$ induced by $\Psi$.  
Each fiber of $\mu$ is a single $T$-orbit, and the Lie algebra stabilizer of
a point of $\mu^{-1}(p)$ for $p \in \circhkhf$ is $V(F)^* = H_F^\perp\subs V^*$.

The space $\M_\cH$ has a decomposition $\sS$ into the pieces
$$S_F := \mu^{-1}(\circhkhf), \;F \in L_\cH.$$  Equivalently,
$S_F$ is the image of the points $(z,w)\in \Psi^{-1}(\hkv)$
for which $z_i = w_i = 0$ if and only if $i \in F$.  
It is easy to see that $S_E \subset \ol{S_F}$
if and only if $F \subset E$ for any flats $E,F$.
Every point in $S_F$ has the same stabilizer $T_F$,
namely the subtorus of $T$ with Lie algebra $V(F)^*$.
Thus if we induce a linear poset structure on the lattice of flats $L_\cH$
by applying the construction of 
Section \ref{sheaves on posets} to the pair $(\M,\sS)$,
we get exactly the one given in Section \ref{arrangements}.


The following result, which is similar to Lemmas 2.4 and 
2.5 of \cite{PW}, describes the local structure of
this decomposition.  Note that for any flat $F$ the
torus $T_F$ is naturally isomorphic to the torus
which acts on $\M_{\cH_F}$, and $T/T_F$ is naturally
isomorphic to the torus which acts on $\M_{\cH^F}$.

\begin{proposition}\label{strata and slices}
 $\sS$ is a $T$-stratification of $\M$; the normal slice
to a point in the stratum $S_F$ can be taken to be 
isomorphic as a stratified $T_F$-space to $\M_{\cH_F}$.
Furthermore, the closure $\overline{S_F}$ of a 
stratum is isomorphic as a stratified $T/T_F$-space
to $\M_{\cH^F}$.
\end{proposition}

\begin{proof}
 By Lemma \ref{stratification remarks}, to show that $\sS$ is
a $T$-stratification it is enough to work ``upstairs'' 
and show that the decomposition of $\Psi^{-1}(\hkv)$ into
the sets $\Psi^{-1}(\circhkhf)$, $F \in L_{\cH}$ is a 
$T^I$-stratification.  To see this, fix $F$ and
take a point $p\in \Psi^{-1}(\circhkhf)$.  
Choose an $\R$-vector subspace $N \subset \hkv_0 = V_0 \times V_0^\C$
which is complementary to $\hkhf$.
Then we can find an open disk $D \subset \hkhf$ 
centered at $p$ and (shrinking $D$ if necessary)  
an open disk $B \subset N$ centered at $0$ so that
$B + D$ meets $\hkh_E$ if and only if $E \le F$. 

Let $U = \Psi^{-1}(B + D)$.  If $\Psi_F\colon \H^F \to (\Im \H)^F$
is the hyperk\"ahler moment map for $\H^F$, then we have a map 
$U \to \Psi^{-1}_F(V(F)^\text{hk})$ 
given by restricting the projection $\H^I \to \H^F$.   Its image is 
$\Psi_F^{-1}(B_F)$, where $B_F \cong B$ is the image of $B$ under the 
projection $\pi_F \otimes 1_{\Im \H}$.  This will be the normal slice
to the stratum $\Psi^{-1}(\circhkhf)$.

Define a map $\tau\colon U \to T^{I\setminus F}$ as follows.  
Suppose that the $i$th coordinate of $p$ is $z_i + w_i\, \mathbf{j}$, where 
$z_i$, $w_i \in \C$.  Then if a point $q \in U$ has coordinates $z'_i + w'_i\, \mathbf{j}$, 
we let the $i$th coordinate of $\tau(q)$ be $z'_i/|z'_i|$ if $z_i \ne 0$
and $w'_i/|w'_i|$ otherwise (we can shrink the disks $B$ and $D$ if
necessary to ensure that $z'_i \ne 0$ for all points in $U$ in the first case, and
$w'_i\ne 0$ in the second case).  

Consider the map $U \to \Psi_F^{-1}(B_F) \times D$ which sends $q$ to
$(\psi(q), d)$, where $\Psi(q) = b + d$, $b \in B$, $d \in D$.  
Its restriction to $\tau^{-1}(1)$ is a continuous, proper bijection,
and so it is a homeomorphism.  If $\sigma$ denotes its inverse, then
$(t, x, y) \mapsto t\cdot \sigma(x, y)$ defines the required
homeomorphism
\[T^I \times_{T^F} \Psi_F^{-1}(B_F) \times D \to U.\]

Since $V(F)$ is a linear subspace of $\R^F$, 
the stratifications of $V(F)^\text{hk}$ and $\Psi^{-1}_F(V(F)^\text{hk})$ are 
invariant under the the multiplication action of
$\R^+$.  Thus we have a stratum-preserving
homeomorphism $\Psi_F^{-1}(B_F) \cong \Psi_F^{-1}(V(F)^\text{hk})$, 
and the normal slice is topologically a cone, as required.  In addition, this
implies that the normal slice to $S_F$ in $\M$ is isomorphic as a $T_F$-stratified
space to $\Psi_F^{-1}(V(F)^\text{hk})/T_{V(F)_0} = \M_{\cH_F}$.

Finally, the identification of $\overline{S_F}$ with $\M_{\cH^F}$ follows
easily by restricting to $\H^{I\setminus F}$.
\end{proof}

\vspace{-\baselineskip}
\begin{corollary}
The stratification $\sS$ of $\M_\cH$
satisfies the conditions used in Section
\ref{localization functor} to define the localization functor.
\end{corollary}

\begin{proof}
We have already shown that it is a topological $T$-stratification. 
The condition (B) follows from the fact that for every flat $F$, the space 
$$S_F/T \cong \circhkhf$$ is a complement of a collection
of codimension three subspaces of $\hkv$, and is therefore simply
connected.  Condition (C) follows from Proposition \ref{strata and slices}
and the fact that strata of hypertoric varieties are connected.
\end{proof}

\vspace{-\baselineskip}
\subsection{Intersection cohomology of hypertoric varieties}
We can now state our first main theorem, which says that the intersection cohomology groups of
the hypertoric variety $\M_\cH$ may be computed in terms of sheaves on $L_\cH$. 
Let $\cL = \cL_\emptyset$ be the minimal
extension sheaf on $L_\cH$ with 
maximal support, and let $\IC_T(\M_\cH)$ be the equivariant intersection 
cohomology sheaf of $\M_\cH$.

\begin{theorem} \label{main theorem 1}
The sheaf $\cL$ is rigid.  There is an isomorphism
\[\Loc \IC_T(\M_\cH) \cong \cL\]
of graded $\cA_\cH$-modules,
and the map $$\Gamma: \IH^\udot_T(\M_\cH)\to\cL(L_\cH)$$ of Equation \eqref{localization homomorphism}
is an isomorphism of graded $\Sym V_0$-modules.
\end{theorem}

\begin{remark}
Since $\cL$ is rigid, the isomorphism of sheaves in Theorem \ref{main theorem 1} is necessarily
unique up to scalar multiplication.  It can be made completely unique using the fact 
that the two graded vector spaces of global sections
are each canonically isomorphic to $\R$ in degree zero.
\end{remark}

We conclude the section with a proposition that will be essential to the proof of
Theorem \ref{main theorem 1}.  It is proved in \cite[\S 2]{PW}, using an alternative,
algebro-geometric construction of hypertoric varieties.

\begin{proposition}\label{resolution}
Let $\cH$ be an arrangement, and $\tcH$ a simplification of $\cH$.
There is a natural semismall projective map $\varpi:\M_\tcH \to \M_\cH$.
This map restricts to an equivariant fiber bundle over each stratum of $\M_\cH$, 
and the fiber over the stratum $S_F$ is $T_F$-equivariantly homotopy equivalent
to $\M_{\wt{\cH_F}}\,$.
\end{proposition}

\begin{remark}\label{core}
In particular, if $\cH$ is central, then the fiber of $\M_\tcH \to \M_\cH$ over the
point stratum of $\M_\cH$ is equivariantly homotopy equivalent to $\M_\tcH$ itself.
This subspace is called the {\bf core} of $\M_\tcH$, and is equivariantly homeomorphic
to a union of toric varieties, one for each bounded chamber of $\tcH$, glued together
along toric subvarieties \cite[6.5]{BD}.  If  
$\M_{\wt\cH}$ is smooth, i.e.\ $\cH$ is unimodular, it is the union of 
open subsets which are isomorphic to the cotangent bundles 
to these toric varieties.
\end{remark}

\begin{remark}
Alternatively, it is possible
to understand {\em} any hypertoric variety, smooth or not, 
in terms of a single toric variety inside of which it sits.
From the definition of $\M_\cH$ in Section \ref{Hypertoric basics}, 
one can see that $\M_\cH$ is a complete
intersection inside of the toric variety $$\Psi^{-1}(V\times \C^I)/T_{V_0} = \Psi_\R^{-1}(V)/T_{V_0}.$$
This toric variety is known as the {\bf Lawrence toric variety} associated to $\cH$.
When $\M_\cH$ is smooth or an orbifold, 
Hausel and Sturmfels \cite[\S 6]{HS} prove that the embedding of $\M_\cH$ into the Lawrence
toric variety induces an isomorphism on cohomology, and use that fact to compute the cohomology
ring of $\M_\cH$.  This method fails, however, when dealing with the more refined invariant of
intersection cohomology of a singular variety.
\end{remark}
\end{section}

\begin{section}{Minimal extension sheaves on the lattice of flats}\label{Lattice of flats}
Our first step toward the proof of Theorem \ref{main theorem 1}
is an algebraic study of the minimal extension sheaves
on the linear poset $L_\cH$.  The results of this section do not 
require the existence of a hypertoric variety, so we
do not need to assume that the arrangement is rational.
In fact, the results are true as stated for arrangements
over an arbitrary field of characteristic zero\footnote{The characteristic zero
hypothesis is needed for \cite[Prop. 7]{PS}, which we use to prove
Proposition \ref{psw}.}, but we will continue to work with
real arrangements so as to keep our notation from Section \ref{hyper} in place.

To simplify matters, we observe that for any flat $F$ of $\cH$, 
the closure $\ol{\{F\}}$ of $\{F\}$ in our topology on $L_\cH$ 
is isomorphic to the lattice of flats 
of the restricted arrangement $\cH^F$, and the restriction of $\cA_\cH$ to $\ol{\{F\}}$ is 
$\cA_{\cH^F} \otimes_{\Sym\langle F\rangle} \Sym V_0$.
It follows that the minimal
extension sheaf $\cL_F$ with support $\ol{\{F\}}$ is obtained by extension
of scalars from the minimal extension sheaf on $\cH^F$ whose support is
all of $L_{\cH^F}$.  Thus we can concentrate on describing $\cL_F$ in
the case of the flat $F = \emptyset$ whose closure is $L_\cH$.

\subsection{Simple arrangements}
Let $\tcH$ be a simplification of $\cH$ in the sense of Section \ref{arrangements};
in other words, the hyperplanes of $\tcH$ are generic translates of the hyperplanes of $\cH$.
Minimal extension sheaves on $L_\tcH$ are easy to understand, 
and this will help us to understand those on $L_\cH$.

There is a natural way besides the one that we have already described
in \S\ref{arrangements} to make the lattice $L_\cH$ into a linear poset: 
associate to the flat $F$ the vector space $\R^F$, with restrictions
given by the obvious quotient maps.  Let $\cA'_\cH$ be the resulting 
sheaf of rings on $L_\cH$.

\begin{lemma}\label{Aprime}
There is a natural homomorphism $\cA_\cH\to\cA'_\cH$ of sheaves of rings, which is an isomorphism
if and only if $\cH$ is simple.  In particular, we have an isomorphism
$\cA_\tcH\cong\cA'_\tcH$.
\end{lemma}

\begin{proof}
The natural inclusions $V(F) \to \R^F$ described in \S\ref{arrangements} 
are compatible with the maps $V(E) \to V(F)$ and $\R^E\to \R^F$ for $E\le F$,
so they induce a homomorphism $\cA_\cH \to \cA_\cH'$.
It will be an isomorphism if and only if each map 
$V(F) \to \R^F$ is an isomorphism, which is equivalent to the simplicity of $\cH$, 
by Remark \ref{simple}.
\end{proof}

The following lemma generalizes Example \ref{first}.

\begin{lemma}\label{Q-smooth MES}
The minimal extension sheaf $\mathcal L_\emptyset$ corresponding
to the maximal flat $\emptyset$ of $L_\tcH$ is isomorphic to the structure sheaf $\cA_\tcH$.
\end{lemma}

\begin{proof}
By Definition \ref{pure} and Proposition \ref{decomposition theorem}, 
we need only show that $\cA_\tcH \cong \cA_\tcH'$ is flabby.  But 
for any $F \in L_\tcH$ the map $\cA_\tcH'(F) \to \cA_\tcH'(\bdy F)$ is 
surjective, with kernel $(\prod_{i\in F} \be_i)\cA_\tcH'(F)$.
\end{proof}

We can also give a nice description of the space of global sections of this sheaf, which,
by Lemma \ref{Q-smooth MES}, is naturally a ring.

\begin{lemma} \label{Sections gives face ring}
The ring of global sections $\cA(L_{\tcH})$ is canonically
isomorphic to the face ring $\R[\Delta_\tcH]$.
\end{lemma}

\begin{proof} 
Because $\tcH$ is simple, we have $\cA_\tcH\cong\cA_\tcH'$ and $L_\tcH = \D_\tcH$.
It is an easy exercise to check that the maps
$\R[\Delta_\tcH] \to \cA_\tcH'(F) = \Sym\R^F$ which send 
$\be_i$ to the $i$th standard basis vector if $i\in F$,
and to zero if $i\notin F$, induce an isomorphism
$\R[\Delta_\tcH] \cong \cA_\tcH'(L_\tcH)$ of graded rings.
\end{proof}

\vspace{-\baselineskip}
\subsection{The pushforward of the structure sheaf}
Recall from Section \ref{arrangements} that we have a natural map of posets 
$\pi\colon L_{\wt \cH} \to L_\cH$ with $\cA_\tcH\cong\pi^*\cA_\cH$.
It follows that
the pushforward $\pi_*$ takes $\cA_{\wt\cH}$-modules to
$\cA_{\cH}$-modules, and we may therefore define
the $\cA_\cH$-module $$\cE := \pi_*\cA_{\wt\cH}.$$

\begin{proposition}\label{E is pure} $\cE$ is a pure $\cA_\cH$-module.
\end{proposition}

\begin{proof} The pushforward of a flabby sheaf is always flabby,
so we need only show $\cE$ is pointwise free.  Let $F\in L_\cH$ be any flat,
and consider the stalk 
$\cE(F) = \cE(U_F) = \cA_{\wt\cH}(\pi^{-1}(U_F))$ of $\cE$ at $F$.
The linear poset $\pi^{-1}(U_F)$ is isomorphic to the linear poset
of the simplification $\wt{\cH_F}$ of $\cH_F$ induced by the simplification
$\tcH$ of $\cH$.  Thus, by
Lemma \ref{Sections gives face ring}, we have
$\cE(U_F) \cong \R[\Delta_{\wt{\cH_F}}]$, which is a free $\cA_\cH(F)$-module
by the discussion of matroid complexes in Section \ref{mat}.
\end{proof}

\vspace{-\baselineskip}
\begin{lemma}
\label{Combinatorially semi-small} 
For any flat $F$ of $\cH$, we have the following three facts.
\begin{enumerate}
\item[(a)] $\cE(F,\bdy F)$ is a free $\cA_\cH(F)$-module generated in
degree $2\rank F$, with a natural basis in one-to-one
correspondence with $\pi^{-1}(F)$.
\item[(b)] The map 
$\ol{\cE(F, \bdy F)} \to \ol{\cE(F)}$
is surjective in degree $2\rank F$ (the top nonzero degree).
\item[(c)] There is an isomorphism 
$\cE\cong \bigoplus_{F\in L_\cH} \ol{\cE(F)}_{2\rank F} \otimes_k
\cL_F$, where we consider $\ol{\cE(F)}_{2\rank F}$ as a graded 
vector space concentrated in degree $2\rank F$.
\end{enumerate}
\end{lemma}

\begin{proof} 
Let $\Delta = \Delta_{\wt{\cH_F}}$ and let
$\Delta^\circ = \Delta_{<\rk F}$ be the subcomplex of $\Delta$ with 
all top-dimensional simplices removed.  We have already 
seen in the proof of Proposition \ref{E is pure} that 
$\cE(F) = \R[\Delta]$.  Since the map $\pi$ preserves rank,
we have $\pi^{-1}(\bdy F) = \{E \in \pi^{-1}(U_F) \mid \rk E < \rk F\}$.
Thus we have $\cE(\bdy F) = \R[\Delta^\circ]$, and the map
$\cE(F) \to \cE(\bdy F)$ is given by the natural restriction 
of face rings.
It follows that $\cE(F,\bdy F)$ has a basis consisting of elements
$e_E = \prod_{i\in E} e_i$ for $E \in \Delta\setminus \Delta^\circ=\pi^{-1}(F)$.
This proves Statement (a).

Next, 
take a shelling order $E_1,\dots,E_r$ of the top-dimensional
simplices of $\Delta$.  Then $\cE(F) = \R[\Delta]$ 
has an $\cA_\cH(F)$-module basis
$e^{}_{S_1}, \dots, e^{}_{S_r}$, where 
\[S_k := \{i\in E_k \mid E_k \setminus \{i\} \subset E_j\; \text{for some $j < k$}\}.\]  
The degree $2\rk F$ elements of this basis are a subset of
our basis for $\cE(F,\bdy F)$, which implies Statement (b).

Last, reduce the short exact sequence
\[0 \to \cE(F,\bdy F) \to \cE(F) \to \cE(\bdy F) \to 0\]
to obtain a right exact sequence
\[\ol{\cE(F,\bdy F)} \to \ol{\cE(F)} \stackrel{\ol{\bdy_F}}\to
\ol{\cE(\bdy F)} \to 0.\]
Statements (a) and (b) imply that $\ker \ol{\bdy_F}$ 
is isomorphic to $\ol{\cE(F)}_{2\rank F}$. 
This fact, along with Propositions \ref{decomposition theorem} and \ref{E is pure}, 
gives us Statement (c).
\end{proof}

\vspace{-\baselineskip}
\begin{remark}
If the arrangement $\cH$ is rational, this decomposition can be
deduced from the geometric decomposition theorem of \cite[6.2.5]{BBD}
along with the 
fact that the map $\varpi\colon\M_{\tcH} \to \M_{\cH}$ is
semismall, as in the toric case (see Remark \ref{decomp theorem remark}).  
As we will see in Proposition \ref{pushforward prop} 
below, $\cE$ is the localization of the pushforward 
$\varpi_*\R_{\M_{\tcH},T}$.  The decomposition theorem says that this
splits into a direct sum of intersection cohomology sheaves, whose
localizations are the various minimal extension sheaves $\cL_F$.
The fact that each sheaf $\cL_F$ can only appear in the decomposition
in a single degree $2\rank F$ follows from the semismallness of 
$\varpi$.

Semi-smallness also implies that the decomposition of 
$\varpi_*\R_{\M_{\tcH},T}$ into IC sheaves is unique.  In contrast,
the isomorphism of Lemma \ref{Combinatorially semi-small}(c)
is not canonical, since it involves choosing
representatives for $\ol{\cE(F)}$ in $\cE(F)$.  The summand
$\cL_\emptyset$, however, {\em does} sit inside $\cE$ canonically.
\end{remark}
\begin{corollary}\label{rigid MES}
The minimal extension sheaf $\cL_\emptyset$ is rigid.
\end{corollary}

\begin{proof}
Lemma \ref{Combinatorially semi-small} tells us that for any flat $F
\ne \emptyset$, the costalk $\cE(F,\bdy F)$ is generated in degree
$2\rank F$, thus so is its summand $\cL_\emptyset(F,\bdy F)$.  The corollary
will follow from Lemma \ref{rigidity criterion} if we can show 
that $\cL_\emptyset(F)$ is generated in degrees strictly less 
than $2\rank F$.

This can be deduced from Corollary \ref{Betti number theorem} below,
but we can also give a simple direct proof.
We observed in Section \ref{mat} that the face ring
$\R[\Delta_{\wt{\cH_F}}]$ is generated as an $\cA(F)$-module in degrees less
than or equal to $2\rank \cH_F = 2\rank F$.  
Since $\cL_\emptyset(F)$ is a summand of
$\cE(F)\cong\R[\Delta_{\wt{\cH_F}}]$, it is also generated in degrees
at most $2\rk F$.  In fact, we can do one better: Lemma
\ref{Combinatorially semi-small}(c) tells us that
$$\cE \cong \cL_\emptyset \oplus \ol{\cE(F)}_{2\rk F}\otimes_k\cL_F\oplus\text{other terms}.$$
Passing to the stalk at $F$, reducing, and taking the part of degree
$2 \rk F$, we obtain the equation
$$\ol{\cE(F)}_{2\rk F} \cong \ol{\cL_\emptyset(F)}_{2\rk F} \oplus\ol{\cE(F)}_{2\rk F}
\otimes_k\ol{\cL_F(F)}_0\oplus\text{other terms}.$$
Since $\ol{\cL_F(F)}_0\cong\R$, we must have
$\ol{\cL_\emptyset(F)}_{2\rk F}=0$, proving our statement.
\end{proof}

\vspace{-\baselineskip}
\subsection{The sheaves \boldmath$\cR^{bc}$ and \boldmath$\cR$}\label{defining R}
In this section we construct two explicit models of the minimal extension
sheaf $\cL$ on $L_\cH$.  As an application we compute the Betti numbers
of the stalk of $\cL$ at every flat, a necessary step toward the proof
of Theorem \ref{main theorem 1}.

Fix an ordering of the indexing set $I$, 
and define a sheaf of rings $\cR^{bc}$ on $L_\cH$ by letting
$$\cR^{bc}(F) := \R[\Delta^{bc}_{\cH_F}]$$ for any flat $F\in L_\cH$,
and letting the map $\cR^{bc}(F) \to \cR^{bc}(E)$ be the restriction 
induced by the inclusion $\Delta^{bc}_{\cH_E} \subset \Delta^{bc}_{\cH_F}$ for 
$E\le F$.  (Note that in fact $\Delta^{bc}_{\cH_E}$ consists of all
simplices of $\Delta^{bc}_{\cH_F}$ which are contained in $E$.)  There is a unique
sheaf with these properties by Proposition \ref{stalks are enough}.

This sheaf $\cR^{bc}$ is a quotient (as a sheaf of rings) of the sheaf
$\cA'$ defined in the proof of Lemma \ref{Sections gives face ring}, since
$\R[\Delta^{bc}_{\cH_F}]$ is a quotient of the polynomial ring $\Sym\R^F$. 
The map $\cA \to \cA'$ makes $\cR^{bc}$ into an algebra over $\cA$,
and in particular into an $\cA$-module.

\begin{lemma}\label{sections over U}
For any open set $U$,
$\cR^{bc}(U)$ is isomorphic to the face ring of 
the complex $\bigcup_{F\in U} \Delta^{bc}_{\cH_F}$.
\end{lemma}

\begin{proof}
Recall from Proposition \ref{stalks are enough} that
$$\cR^{bc}(U) = \ilim_{F\in U} \cR^{bc}(F).$$
It is clear that the face ring of $\bigcup_{F\in U} \Delta^{bc}_{\cH_F}$ maps surjectively
to each $\cR^{bc}(F)$, and therefore to the inverse limit.  
To see that this map is an isomorphism, note that
$\cR^{bc}(F)$ has a vector space basis consisting of 
monomials in the $e_i$ whose support $S$ is contained in $F$ and 
contains no broken circuit.  We call such a set $S$ {\em allowable}
for $F$.  Suppose a monomial has support $S$ which is allowable for
two flats $F_1$ and $F_2$.  Then for any section
of $\cR^{bc}$ on a set $U$ containing both flats 
the coefficients of this monomial on $F_1$ and $F_2$
must be the same, since $S$ is allowable for $F_1 \cap F_2 \in U$.
\end{proof}

\vspace{-\baselineskip}
\begin{proposition}\label{broken circuit MES}
The sheaf $\cR^{bc}$ is a minimal extension sheaf.
\end{proposition}

\begin{proof}
The fact that $\cR^{bc}$ is pointwise free follows from the fact that
broken circuit complexes are Cohen-Macaulay, and flabbiness is a consequence
of Lemma \ref{sections over U}.
Thus we only need to show that $\cR^{bc}$ is indecomposable as an $\cA$-module.  
By Proposition \ref{decomposition theorem}, it is enough to show that 
for any $F \ne \emptyset$, the map
$\ol{\cR^{bc}(F, \bdy F)} \to \ol{\cR^{bc}(F)}$
is zero.
The target $\ol{\cR^{bc}(F)}$ is zero in degrees $\ge 2\rank F$, since
$\deg h_{\cH_F}^{bc}(q) < \rank \cH_F = \rk F$.  On the other hand,
$\ol{\cR^{bc}(F, \bdy F)}$ is concentrated in degree $2\rank F$, 
since $\displaystyle\bigcup_{E < F} \!\!\Delta^{bc}_{\cH_E}$ is the 
$(\rank F - 2)$-skeleton of $\Delta^{bc}_{\cH_F}$.  Thus the map vanishes for degree reasons.
\end{proof}

\vspace{-\baselineskip}
\begin{corollary} \label{Betti number theorem}
For any arrangement $\cH$ and any flat
$F\in L_\cH$, the Hilbert series of the reduction of the stalk of $\cL_\emptyset$
at $F$ coincides with the $h$-polynomial of the broken circuit complex
of the localized arrangement $\cH_F$ with degrees doubled:
\[\Hilb(\ol{\cL_\emptyset(F)}, t) = h_{\cH_F}^{bc}(t^2).\]
\end{corollary}

We have now constructed an explicit minimal extension sheaf $\cR^{bc}$ on $\cL$.
There is, however, something unsatisfying about this construction.
Our definition of the sheaf $\cR^{bc}$ depends on the notion of a broken circuit,
which in turn depends on an ordering of the indexing set $I$.
We would rather have a construction that does not depend on such a choice.

For $\cH$ a central arrangement, David Speyer and the second author defined a ring
$R(\cH)$ that doesn't depend on any choices, and admits a flat degeneration
to the face ring $\R[\Delta^{bc}_\cH]$ for any choice of ordering of $I$ \cite[Thm. 4]{PS}.
This ring is defined to be the subring of rational functions
on $V$ generated by $x_i^{-1}$, $i\in I$, where $x_i$ is the restriction of 
the $i^{\text{th}}$ coordinate function to $V \subset \R^I$.
More explicitly, if we put $\be_i = x_i^{-1}$, the ring has a presentation of the form
$$R(\cH) = \R[\be_i \mid i\in I]\left/\left< \sum_{i\in I_0}a_ie^{}_{I_0 \smallsetminus \{i\}} 
\bigmid I_0 \subset I\;\, \text{and}\;\,\sum_{i\in I_0} a_i x_i|_V = 0\right>\right..$$
(Note that this ideal is generated by those terms for which $I_0$ is a circuit.)
We now define a sheaf of $\cA$-algebras
$\cR$ on $L_\cH$ by putting $$\cR(F) := R(\cH_F)$$ for all $F\in L_\cH$, with restriction
maps $\cR(F)\to\cR(G)$ given by setting the variables
$\{\be_i\mid i\in F\smallsetminus G\}$ to zero.  
The $\cA(F)$-module
structure on $\cR(F)$ is given by the inclusion $V(F) \hookto \R^F$ defined in 
Section \ref{arrangements}.
Note that $\cH_F$ is always central, so the fact that the ring $R(\cH)$ is only defined
for central arrangements does not pose any problems.

\begin{proposition}\label{psw}
The sheaf $\cR$ is a minimal extension sheaf.
\end{proposition}

\begin{proof}
We need to show that the stalks of $\cR$ are free modules over the stalks
of $\cA$, that $\cR$ is flabby, and that it is indecomposable as an $\cA$-module.
The first statement is proved in \cite[Prop. 7]{PS}.  
To prove the second and third statements, we will show that $\cR$ is isomorphic
to $\cR^{bc}$ as a sheaf of graded vector spaces.  Since we have already shown
in Proposition \ref{broken circuit MES} that $\cR^{bc}$ is flabby and that it is
indecomposable {\em for degree reasons}, this will be sufficient.

By \cite[Thm. 4]{PS} the ring $\cR(F)$ has an $\R$-basis given by 
all monomials in the variables $\be_i$
whose support is contained in $F$ and contains no broken circuit.
We have already observed that these monomials form a basis for the ring
$\cR^{bc}(F)$, hence they may be used to define a vector space isomorphism
$\psi_F\colon \cR^{bc}(F) \to \cR(F)$ for each $F$.
These isomorphisms are compatible with
the restriction maps, and therefore define an isomorphism of sheaves
of graded vector spaces.
\end{proof}

\vspace{-\baselineskip}
\begin{remark}\label{how canonical}
Propositions \ref{broken circuit MES} and \ref{psw} say that the sheaves
$\cR^{bc}$ and $\cR$ are both minimal extension sheaves on $L_\cH$,
and are therefore canonically isomorphic as sheaves of $\cA$-modules
by Corollary \ref{rigid MES}.
They are \emph{not}, however, isomorphic as sheaves of $\cA$-algebras.
Propositions \ref{broken circuit MES} and \ref{psw}, when coupled with
Theorem \ref{main theorem 1}, give two different ways to put a ring structure on the
equivariant intersection cohomology group $\IH^\udot_T(\M_\cH)$.
For now, we will say that only the ring structure coming from the sheaf
$\cR$ is canonical, because this did not involve making any unnatural choices.
Later we will give this assertion more precise mathematical meaning by showing
that, if $\cH$ is unimodular and central, the group equivariant IC sheaf of $\M_\cH$ admits
a unique ring structure in the equivariant derived category, and the induced
ring structure on $\IH^\udot_T(\M_\cH)$ is the one coming from $\cR$.
\end{remark}

\begin{example}\label{asymmetric}
 We illustrate the sheaves $\cR$ and $\cR^{bc}$ using a central arrangement 
of three lines in the plane.  More precisely, let $V  = V_0 \subset \R^3$ be 
the vector subspace spanned by $v_1 = (1, 0, 1)$ and $v_2 = (0, 1, 1)$.
The resulting linear poset is the one described in Example \ref{second}, with 
$V(1) = V/\R\cdot v_2$, $V(2) = V/\R\cdot v_1$, and $V(3) = V/\R\cdot (v_1 - v_2)$.
For all flats $F$ except the maximal flat $\{1, 2, 3\}$, we have 
\[\cR(F) = \cR^{bc}(F) \cong \cA(F) = \Sym V(F),\]
since $\cH_F$ is simple for these flats.

At the flat $F =\{1, 2, 3\}$, however, we see a difference between $\cR$ and $\cR^{bc}$. 
We have $\cR(F) = \R[e_1, e_2, e_3]/\la e_2e_3 + e_1e_3 - e_1e_2\ra$, 
while $\cR^{bc} = \R[e_1, e_2, e_3]/\la e_2e_3\ra$; these are clearly not
isomorphic as rings.  The sheaves $\cR$ and $\cR^{bc}$ {\em are} 
canonically isomorphic as modules over $\cA$, since they are
both minimal extension sheaves.  This gives an isomorphism
$\cR(F) \simeq \cR^{bc}(F)$ of $\cA(F) = \Sym(V)$-modules; it is
the unique isomorphism which induces the identity on the spaces
of linear forms.
\end{example}

\section{The hypertoric IC sheaf}
We now prove Theorem \ref{main theorem 1} and explore some of its consequences.

\subsection{Proof of Theorem \ref{main theorem 1}}
Let $\cH$ be an arbitrary rational arrangement, and $\tcH$ a simplification of $\cH$.
Let $$\M := \M_\cH \,\,\,\,\text{and}\,\,\,\,\wt\M := \M_\tcH.$$
Since $\tcH$ is simple, $\wt\M$ is rationally smooth, and
$\IC_{\wt\M,T} = \R_{\wt\M,T}$.
Our strategy will be to first prove the theorem for $\tcH$,
and then tackle the general case by
exploiting the relationship between the maps
$$\pi:L_\tcH\to L_\cH$$ of Section \ref{arrangements} and 
$$\varpi:\M_\tcH\to \M_\cH$$ of Proposition \ref{resolution}.  

\begin{proposition}\label{simple case}
There is an isomorphism
$\Loc \R_{\wt\M,T} \cong \cA_\tcH$
of graded $\cA_\tcH$-modules,
and the map $$\Gamma: H^\udot_T(\wt\M) \to \cA_\tcH(L_\tcH)$$ 
of Equation \eqref{localization homomorphism}
is an isomorphism of graded $\Sym V_0$-modules.
\end{proposition}

\begin{proof}
The stalk of $\Loc \R_{\wt\M,T}$ at a point $x\in L_\tcH$ is the $T$-equivariant
cohomology of any fiber of $S_x \to S_x/T$, which is isomorphic to $\cA(x)$.
The restriction maps in $\Loc \R_{\wt\M,T}$ are module maps over $\Sym V_0$,
thus to see that they are the natural quotient maps 
it is enough to know that they are isomorphisms in degree zero. 
This follows from the existence of the map $\Gamma$.

It remains only to show that $\Gamma$ is an isomorphism.
Lemma \ref{Sections gives face ring} tells us that $\cA(L_\tcH)$
is isomorphic to the face ring $\R[\Delta_\tcH]$.
The fact that $H^\udot_T(\wt\M)$ is also isomorphic to $\R[\Delta_\tcH]$
can be inferred from \cite{Ko} in the case where $\cH$ is unimodular,
and from \cite[1.1]{HS} in the general case.  For an explicit statement and proof
of this result, see \cite[3.2.2]{Pr}.  The values of $\Gamma$ on the generators
of $H^\udot_T(\wt\M)$ can be computed via the stalk maps, from which we can conclude
that $\Gamma$ is an isomorphism.
\end{proof}

The following proposition says that localization commutes with pushing forward.

\begin{proposition} \label{pushforward prop} 
There is a natural isomorphism of $\cA_\cH$-modules
$$
\pi_*\cA_{\wt\cH} \cong \Loc \varpi_*\R_{\wt\M,T}.
$$\end{proposition}

\begin{proof}
Let $\cE = \pi_*\cA_{\wt\cH}$ and $\cF = \Loc \varpi_*\R_{\wt\M,T}$.
As was explained in the proof of Proposition \ref{E is pure}, 
the stalk of $\cE$ at a flat $F$ is isomorphic to the 
face ring $\R[\Delta_{\wt{\cH_F}}]$.  The restriction map
$\cE(E) \to \cE(F)$ for $E \le F$ sends a generator $\be_i$, $i\in E$
to the corresponding generator of $\cE(F)$ if $i \in F$, and to $0$
otherwise.

On the topological side, the stalk of $\cF$ at a flat $F\in L_\cH$ is 
the equivariant cohomology 
\[H^\udot_T(\varpi^{-1}(Tp))\cong H^\udot_{T_F}(\varpi^{-1}(p)),\] 
where $p$ is any point in the stratum $S_F \subset \M$
and $T_F\subs T$ is the stabilizer of $p$.  
Proposition \ref{resolution} tells us that
$\varpi^{-1}(p)$ is $T_F$-equivariantly homotopy equivalent to
$\M_{\wt{\cH_F}}$, thus the stalk $\cF_F$ is isomorphic to 
$H_{T_F}^\udot(\M_{\wt{\cH_F}})$, which in turn is isomorphic to
$\R[\Delta_{\wt{\cH_F}}]$ by \cite[3.2.2]{Pr}.  
What remains 
is to show that the restriction maps and the 
$\cA_\cH$-module structure are the same as those of $\cE$.

For any $i\in F$, let $S_i$ be the stratum of $\M_{\wt{\cH_F}}$ corresponding
to the singleton flat $\{i\}$.
The isomorphism of \cite[3.2.2]{Pr} identifies the generator $e_i\in \R[\Delta_{\wt{\cH_F}}]$ 
with a class in $H_{T_F}^\udot(\M_{\wt{\cH_F}})$
which restricts to a nonzero class on any
$T_F$-orbit in $S_{\{i\}}$, but restricts to zero on any 
$T_F$-orbit in $S_{\{j\}}$ for $j\neq i$.  It follows that, for any $E\leq F\in L_\cH$, our
restriction maps 
$$\R[\Delta_{\wt{\cH_F}}]\cong\cE(F) \to \cE(E) \cong \R[\Delta_{\wt{\cH_E}}]$$
are given by setting $e_i$ to zero for all $i\in F\smallsetminus E$,
which agree with the maps
$$\R[\Delta_{\wt{\cH_F}}]\cong\cF(F) \to \cF(E) \cong \R[\Delta_{\wt{\cH_E}}]$$
given in the proof of Proposition \ref{E is pure}.  The agreement of the $\cA_\cH$-module
structures may be verified in a similar manner.
\end{proof}

The final ingredient to the proof of Theorem \ref{main theorem 1} is the following result of
\cite[4.3]{PW}.

\begin{theorem}\label{ihbetti}
If $\cH$ is a central arrangement defined over the rational numbers,
then the intersection cohomology Poincar\'e polynomial
of $\M_\cH$ coincides with the $h$-polynomial 
of the broken circuit complex with degrees doubled:
$$\Hilb(\IH^\udot(\M_\cH),t) = h_\cH^{bc}(t^2).$$
\end{theorem}

\begin{corollary}\label{ic stalk betti}
Let $\cH$ be any arrangement defined over the rational numbers, and let $F$ be a flat of $\cH$.
For any point $x_F$ on the stratum $S_F$, the Hilbert series of the cohomology
of the stalk of the $\IC$ sheaf at $x_F$ coincides with the $h$-polynomial of the broken circuit complex
of the localized arrangement $\cH_F$ with degrees doubled:
$$\Hilb(H^\udot_{x_F}(\IC_{\M_\cH}),t) = h_{\cH_F}^{bc}(t^2).$$
\end{corollary}

\begin{proof}
This follows immediately from Proposition \ref{strata and slices} and Theorem \ref{ihbetti},
the latter applied to the central arrangement $\cH_F$.
\end{proof}

The decomposition theorem for perverse sheaves of \cite{BBD}
tells us that $\Loc \IC_T(\M)$ is a summand of
$\Loc \varpi_*\R_{\wt\M,T}$, and therefore of $\cE$ by Proposition
\ref{pushforward prop}.
Using the computations
of the stalk Betti numbers of the minimal extension
sheaf $\cL$ (Corollary \ref{Betti number theorem}) 
and of the intersection cohomology sheaf (Corollary \ref{ic stalk betti}),
we see that this summand must be $\cL$, in its 
unique embedding into $\cE$ described in Lemma \ref{Combinatorially semi-small}(c).

To see that the map $$\Gamma: \IH^\udot_T(\M)\to\cL(L_\cH)$$ is an isomorphism,
consider the class of objects $B \in D^b_{T,\sS}(\M)$ 
for which $$\GB:H^\udot_T(\M; B) \to (\Loc B)(L_\cH)$$ is an isomorphism.  
Lemma \ref{Sections gives face ring} and Proposition \ref{pushforward prop} 
combine to show
that $\varpi_*\R_{\wt\M,T}$ belongs to this class.  
Since this class of sheaves is clearly closed under taking summands, 
$\IC_T(\M)$ also belongs to it.  This completes the proof of
Theorem \ref{main theorem 1}.

As a consequence of Theorem \ref{main theorem 1} and Proposition \ref{psw}
we obtain the following corollary, which was conjectured in \cite[6.4]{PW}.

\begin{corollary}\label{it's a ring}
There is a natural isomorphism of graded vector spaces $\IH^\udot_T(\M)\cong R(\cH)$.
\end{corollary}

The second half of that conjecture, which states that the ensuing ring structure
on $\IH^\udot_T(\M)$ may be interpreted as an intersection pairing, is the subject
of Section \ref{lifting ring structure}.

\subsection{The Morse stalk}
Fix a rational arrangement $\cH$, and let $L = L_\cH$ and $\M = \M_\cH$.
Although Theorem \ref{main theorem 1} gives a topological 
interpretation for the global sections of a minimal extension sheaf on $L$,
it is more difficult to understand the sections on other 
open sets.  In this section we study the space of sections over a particular open
set, and discover its topological meaning.  
The results of this section are not needed in the rest of the paper;
we include them because they will be of fundamental importance
in the forthcoming paper \cites{BLPW}.

A naive guess would be that the sections on 
an open set $U \subset L_\cH$ give the intersection cohomology
of $\bigcup_{F \in U} S_F$, but this is {\em not} correct.
Too much information has been 
lost by the localization functor, since by \eqref{localized stalks} it 
essentially treats each stratum as if it were a single $T$-orbit.
For instance, $\cL(\{\emptyset\}) = \R$, which is not
isomorphic to $$H^\udot_T(S_\emptyset) \cong H^\udot(S_\emptyset/T) \cong
H^\udot\!\left(\hkv\smallsetminus\bigcup_{i \in I} \hkhi\right).$$ 
Later, in Section \ref{lifting ring structure},
we will introduce a more refined localization functor 
which uses the cohomology along entire strata and which
can capture the intersection cohomology of any open union of strata.

There is one other case in which we can understand
the space of sections $\cL(U)$, namely when 
$U = L_{>0}$ is the set of flats of corank $> 0$, 
corresponding to strata of positive dimension. 
We will interpret the spaces $\cL(L_{>0})$ and $\cL(L,L_{>0})$
in terms of the intersection cohomology of certain subspaces
of $\M$, as follows.  

We begin by recording some more properties of hypertoric varieties.
The algebro-geometric construction to which we alluded in Section \ref{Hypertoric basics}
allows us to extend the action of $T$ on $\M$ to an algebraic action of the
complexified torus $T_\C$.  This action is hamiltonian with respect to a complex
symplectic form on $\M$, with moment map $$\mu^{}_\C:\M \to \t^*_\C \cong V_0^\C$$
obtained by composing the hyperk\"ahler moment map
$\mu:\M\to\hkv$ with the projection
$\hkv = V \times V_0^\C \to V_0^\C$.
For any polyhedron $\Delta\subs V_0$, let $X_\Delta$ be the associated
toric variety along with its natural $T_\C$-action.
The following proposition first appeared in \cite[\S 2]{HP}.

\begin{proposition}
The subvariety $\mu_\C^{-1}(0)$ of $\M$ is $T_\C$-equivariantly
isomorphic to a union of toric varieties $X_{\Delta}$ indexed by the chambers
of $\cH$, where $X_{\Delta}$ is glued to $X_{\Delta'}$ along the toric subvariety
$X_{\Delta\cap\Delta'}$.
\end{proposition}

\begin{remark}
The subvariety $\mu_\C^{-1}(0)$ is called the {\bf extended core} of $\M$.
Sitting inside the extended core is the ordinary core, which has already been
mentioned (when $\cH$ is simple) in Remark \ref{core}.  The ordinary
core may be defined as the union of all compact toric subvarieties of the extended core.
If at least one chamber $\Delta$ is bounded, then the core may be described
as the union of those $X_{\Delta}$ for which $\Delta$ is bounded.
But this is not always the case.  For example if $\cH$ is central, then every chamber
is unbounded, and the core consists of a single point.
\end{remark}

Fix a cocharacter $\rho\colon U(1) \to T$
for which $\M^{\rho(U(1))} = \M^T$, and let $\rc\colon\C^*\to T_\C$ denote its complexification.
Define two subspaces $\M^\pm$ of $\M$ by
\[\M^{\pm} := \left\{x \in \M \bigmid \lim_{t\to 0} 
\rc(t)^{\pm 1} x \; \text{exists in $\M$}\right\}.\]
Both of these subspaces are unions of extended core components of $\M$.
Specifically, we can view the derivative $d\rho$ of $\rho$
as an element of $\mathop{\mathrm{Lie}} T = V^*$, and $\M^+$ (respectively $\M^-$)
is the union of those $X_\Delta$ for which $d\rho$ has a minimum
(respectively a maximum) on $\Delta$.  The intersection 
$\M^0 = \M^+ \cap \M^-$ is the ordinary core of $\M$.

The space $\M \setminus \M^+$ inherits a stratification
by restricting the strata from $\M$.  All the
strata intersect non-trivially with $\M \setminus \M^+$
except for the zero-dimensional strata, so the poset of
strata is $L_{>0}$.  We would like to consider the localization functor associated
with this stratified space, but to do so we must first check that conditions (A) through (C)
of Section \ref{localization functor} are satisfied.  Conditions (A) and (C) are local, 
so they are preserved when we pass to the open subset $\M \setminus \M^+$
of $\M$.  For condition (B), 
suppose that $\dim S_F \ne 0$.  Then the quotient
$(S_F \setminus (S_F\cap \M^+))/T$ may be identified (via the hyperk\"ahler moment map) with a subset of $H^{\text{hk}}_F$.
The part of $H^{\text{hk}}_F$ that we do {\em not} get includes
the subspaces $\hkhi \cap H^{\text{hk}}_F$, $i\notin F$, which have real codimension three,
along with a finite number of
polyhedral subspaces of real codimension $2\rank \cH^F$.
If $\rank \cH^F > 1$, then we are done. If
$\rank \cH^F = 1$, we get the complement of 
a finite number of line segments and a ray joined end-to-end in 
$\R^3$, so again the space is simply connected. 

Thus we have a localization functor 
$\Loc\colon D^b_{T,\sS}(\M\setminus \M^+) \to (\cA|_{L_{>0}})\md$.  
It is easy to see that restricting from $\M$ to $\M\setminus \M^+$ 
and then localizing is the same as localizing and then 
restricting to $L_{>0}$.  It follows that for any $B \in D^b_{T,\sS}(\M)$,
if we put $\cB = \Loc B$, then we have a commutative diagram with exact rows
\begin{equation}\label{rho-localization diagram}
\xymatrix{& H^\udot_T(\M, \M\setminus\M^+; B) \ar[r]\ar[d] & H^\udot_T(\M; B) \ar[r]\ar[d] & 
  H^\udot_T(\M\setminus\M^+; B) \ar[d]\\
0 \ar[r] & \cB(L,L_{>0}) \ar[r] & \cB(L)\ar[r] & \cB(L_{>0})
}  
\end{equation}
where the upper row comes from the long exact sequence of the pair
$(\M, \M\setminus \M^+)$.  
Note that $H^\udot_T(\M,\M\setminus\M^+; B) = H^\udot_T(\M^+; j^!B)$, where
$j \colon \M^+ \hookrightarrow \M$ is the inclusion.

\begin{definition}
We say that $B$ satisfies $\rho$-localization if 
$H^\udot_T(\M; B) \to H^\udot_T(\M\setminus\M^+; B)$
is surjective and the vertical 
maps in \eqref{rho-localization diagram} are isomorphisms.  
\end{definition}

\begin{theorem}\label{Morse group}
The intersection cohomology sheaf $\IC_T(\M)$ satisfies $\rho$-localization.   
\end{theorem}

\begin{proof}
This property is preserved under taking direct summands, so it is 
enough to show that $\varpi_*\R_{\wt\M,T}$ satisfies 
$\rho$-localization.  Since $\varpi$ is proper and we have the identities
$$\wt\M^+ = \varpi^{-1}(\M^+)\,\,\,\text{and}\,\,\,\pi^{-1}(L_{>0}) = \wt{L}_{>0},$$
this is equivalent to showing that the sheaf $\R_{\wt\M,T}$ on $\wt\M$
satisfies $\rho$-localization.
Thus we have reduced the proof to the case where $\cH$ is simple,
so $\M$ is rationally smooth and $\IC_T(\M)$ is the equivariant constant sheaf.

In this case, \eqref{rho-localization diagram} becomes the following diagram.
\[
\xymatrix{& H^\udot_T(\M, \M\setminus\M^+) \ar[r]\ar[d] & H^\udot_T(\M) \ar[r]\ar[d] & 
  H^\udot_T(\M\setminus\M^+) \ar[d]& \\
0 \ar[r] & \cA(L,L_{>0}) \ar[r] & \cA(L)\ar[r] & \cA(L_{>0})\ar[r] & 0
}  
\]
The bottom row is right 
exact because $\cA$ is flabby.  By Theorem \ref{main theorem 1}
the middle vertical map is an isomorphism.
It is enough to show that the left vertical map
is an isomorphism, since this will imply the vanishing of the connecting homomorphisms for
the pair $(\M, \M\setminus \M^+)$, and the five-lemma will take care of the rest.

Let $\M^T = \{x_1,\dots,x_r\}$,
and let $F_i$ denote the flat of $\cH$ for which $S_{F_i} = \{x_i\}$.
Let $$C_i := \{x \mid \lim_{t\to 0} \rho(t)x = x_i\},$$ so that $\M^+ = \bigcup_{i=1}^r C_i$.
Each $C_i$ is an $2d$-dimensional rational homology cell
(a quotient of $\C^d$ by a finite group).  Choose an ordering
of the $x_i$ so that $\M_i = \M \setminus \bigcup_{j > i} C_j$ is
open for all $i$; thus $\M_0 = \M\setminus \M^+$ and $\M_r = \M$.  
By excision, $H^\udot_T(\M_i, \M_{i-1})$ is isomorphic to 
$H^\udot_T(D_i, D_i \setminus C_i)$, where $D_i$ is a quotient
of $\C^{2d}$ by a finite group and the inclusion $C_i \hookrightarrow D_i$
comes from a coordinate inclusion $\C^d \hookrightarrow\C^{2d}$. 
Thus it is a free $H^\udot_T(pt)$-module of rank 1, generated in degree $2d$.

It follows that $H^\udot_T(\M,\M\setminus\M^+)$ is a free $H^\udot_T(pt)$-module
of rank $r$ generated in degree $2d$, and that the map to 
$H^\udot_T(\M)$ is an inclusion.  This implies that the left 
vertical map is injective, and since $\cA(L, L_{>0})$ is 
also free of rank $r$ and generated in degree $2d$,
this proves the theorem. 
\end{proof}

\vspace{-\baselineskip}
\begin{remark}
When $\cH$ is a central arrangement, so the hypertoric variety
$\M_\cH$ has a unique $T$-fixed point $p$, the functor
$B \mapsto H^\udot_T(\M, \M\setminus\M^+; B)$ is the cohomology
of an equivariant version of the hyperbolic localization
functor explored in \cite{Br}.  It can also be viewed as
the Morse group of $B$ for the stratified Morse function
obtained by composing $d\rho\in V^*$ with the real moment
map $\mu_\R\colon\M_\cH \to V$.

Essentially the same functor was used by Mirkovic and Vilonen
in \cite{MV} to give the weight space in their construction
of the geometric Satake correspondence identifying IC sheaves
of Schubert varieties in the loop Grassmannian of a reductive group $G$ 
with representations of the 
Langlands dual group $G^\vee$.  They were able to prove the
main properties of this functor without using the decomposition
theorem; since the singularities of hypertoric varieties have
many similarities to singularities of loop Grassmannian
Schubert varieties, we hope that similar ideas may make it
possible to prove Theorem \ref{main theorem 1} and Theorem
\ref{Morse group} without using the decomposition theorem.
\end{remark}
\end{section}

\section{Lifting the ring structure to the derived category} 
\label{lifting ring structure}
Let $\cH$ be a unimodular central arrangement.
In \cite[6.4]{PW}, it was conjectured that there is a natural isomorphism of $\Sym V_0$-modules between $\IH_T^\udot(\M)$
and $R(\cH)$, and that this isomorphism may be interpreted as an intersection pairing.  In
Corollary \ref{it's a ring} we have already produced such a natural isomorphism, but it is not clear what our
isomorphism has to do with intersection theory.  
This is partly explained by the following theorem, which is the main result of the remainder of the paper.

Let $\M = \M_\cH$ be a hypertoric variety defined by a {\em unimodular} central
arrangement $\cH$.  
Let $\IC = \IC_{\M,T}$ be the equivariant intersection cohomology sheaf, and let
$u:\R_{\M,T}\to\IC$ be the natural map.

\begin{theorem}\label{IC is ring object}
The object $\IC$ can be made into a commutative ring object in $D^b_T(\M)$ with unit $u$.
More precisely, there is a commutative and associative
morphism 
\[m\colon \IC \otimes \IC \to \IC\]
such that the natural map 
$$\IC \,\,\cong\,\, \R_{\M,T} \otimes \IC \,\overset{u\otimes id}{\longrightarrow}\, 
\IC \otimes \IC \,\overset{m}{\longrightarrow}\, \IC$$
is the identity.  This ring structure is unique.

Applying the localization functor
$\Loc$ to this ring structure gives the ring structure on the minimal extension sheaf $\cL$ 
coming from Theorem \ref{main theorem 1} and Proposition \ref{psw}.  
In particular, the ring structure induced by $m$ on $\IH_T^\udot(\M)$ is that of Corollary \ref{it's a ring}.
\end{theorem}

\begin{remark}
A ring structure on an object in the derived category
induces a ring structure on $H^\udot_T(Y; \IC|_Y)$ for any
$T$-invariant subspace $Y \subset M$, and it does so in a functorial way.  In particular, the
restriction map
\begin{equation}\label{restriction to open stratum}
R(\cH) \cong \IH^\udot_T(\M) \to H^\udot_T(S_\emptyset)
\end{equation}
is a ring homomorphism, where the target has the usual
ring structure.  Thus Theorem \ref{IC is ring object} provides
another, deeper sense in which our ring structure can be called
natural, at least for unimodular arrangements.
In contrast, the analogous natural map from
$R^{bc}(\cH)$ to $H^\udot_T(S_\emptyset)$ is {\em not} a ring homomorphism.
\end{remark} 

\begin{remark}
The fact that we need to assume $\cH$ is unimodular 
is somewhat puzzling, and deserves some explanation.  The
intersection cohomology groups we are calculating are
taken with rational coefficients, and their dimensions
are independent of any lattice structure.  In fact, 
the results of Section \ref{Lattice of flats} 
hold over an arbitrary field of characteristic zero,
even if there is no hypertoric variety in the picture.
This is similar to the situation for toric varieties, where
a theorem of Karu \cite{Ka} (see also \cite{BreLu2}) can be used 
to show that statements about intersection cohomology or even 
more general equivariant perverse sheaves on toric varieties 
can be proved for fans without any rationality hypothesis (see \cite[\S5 and \S6]{BraLu} for example).
 
However, something unexpected happens for hypertoric varieties.  
The action of $T$ on the open stratum $S_\emptyset$ is
quasi-free, and the quotient $S_\emptyset/T$ is homeomorphic to the 
complement in $\hkv$ of the codimension $3$ subspaces
$\hkhi$ (see the notation in \S\ref{Hypertoric basics}).
The cohomology of this space is given by \cite[5.6]{dLS}
as a quotient of a polynomial ring where among the relations
are $\sum_{i \in C} \sgn(a_i)\be^{}_{C\setminus \{i\}}$ where
$\sum_{i\in C} a_i x_i|_V = 0$ is the relation for a circuit $C$.  For 
general $\cH$ this
differs from the relation $\sum a_i\be^{}_{C\setminus\{i\}}$ which 
holds in the ring $R(\cH)$, but it is the same if $\cH$ is unimodular.  This means
that \eqref{restriction to open stratum} cannot be a ring homomorphism
(with the ring structure we have described on $\IH^\udot_T(\M)$)
if $\cH$ is not unimodular, and so Theorem \ref{IC is ring object} cannot
hold.  Note that for some non-unimodular arrangements 
there may still exist a ring structure
on $\IC$, 
but the resulting ring structure on
$\IH^\udot_T(\M)$ will not agree with that of Corollary \ref{it's a ring}.

We do not have a good explanation for this situation, but since 
non-unimodularity of $\cH$ is equivalent to the orbifold resolution
$\M_{\wt\cH}$ having singular points, we speculate that there may
be corrections coming from orbifold cohomology which would make
a statement like Theorem \ref{IC is ring object} possible for
general rational arrangements.
\end{remark}

Our strategy for proving Theorem \ref{IC is ring object} is to 
compute the multiplication map $m$ in terms of its localization in
a combinatorial category of sheaves, and then lift this combinatorial
multiplication to a derived category morphism.  We already have
a ring structure on $\Loc \IC_{\M,T}$, but the functor $\Loc$
is not fully faithful, so we cannot lift this ring structure to
the derived category.  The problem can be seen, for instance,
by noting that the decomposition
of $\cE = \Loc \varpi_*\R_{\M_{\wt\cH,T}}$ is not canonical, whereas
the decomposition of $\varpi_*\R_{\M_{\wt\cH,T}}$ into intersection
cohomology sheaves is canonical, since $\varpi\colon \M_{\wt\cH}
\to \M_\cH$ is semismall.  

To solve this problem we construct a richer localization
$\Loch$ which takes cohomology of a derived category object along
entire strata rather than single $T$-orbits.  
We describe this
functor in Section \ref{refined localization} and show that it completely captures 
homomorphisms between objects satisfying a parity vanishing condition.
In Section \ref{GMES} we describe $\Loch(\IC)$, and show that under suitable 
hypotheses it is a ``generalized minimal extension sheaf'' or GMES, and 
that a ring structure on a GMES
induces a ring structure on $\IC$.  All of this is done
for a general stratified $T$-space satisfying certain parity vanishing
conditions; we hope that there will be other interesting spaces
satisfying our hypotheses.  In Sections \ref{hypertoric structure sheaf} and 
\ref{hypertoric GMES} we work out the specific case
of hypertoric varieties and construct the required
ring structure, thus proving Theorem \ref{IC is ring object}.

\subsection{Stratum-by-stratum Homs in derived categories}
\label{stratum-by-stratum}
Computing Hom-spaces between objects in the derived category 
(or equivariant derived category) is often difficult, in part 
because the derived category is not a stack: in general morphisms cannot be 
built from local data.  In some special cases, however, morphisms can
be localized.  This happens for instance in the subcategory of 
perverse sheaves.   
We will focus on another such case, in which the objects satisfy strong 
parity vanishing conditions along strata.  We show that homomorphisms
between such objects can be described by their restriction to strata.

We start with a connected $T$-space $\M$, 
endowed with a finite $T$-decomposition $\sS$.  
For any locally closed union $\N$ of strata, let $j_\N\colon \N \to \M$ denote
the inclusion.
We make the following assumptions on $\sS$.  The first two are exactly the
same as the first two assumptions from Section \ref{localization functor}. 
\begin{itemize}
\item[(A)] $\sS$ is a $T$-stratification in the sense of Definition \ref{T-stratification}.
\item[(B)] For all $S \in \sS$, the quotient $S/T$ is simply connected (so $T$-equivariant 
local systems on $S$ are trivial).
\item[(C')] For any $R, S\in \sS$ with $R \subset \ol{S}$, we have
$H^k_T(R; j_{R}^*j_{S*}\R_{S,T}) = 0$ for $k$ odd.
\end{itemize}

In particular, when $R=S$,
(C') says that the odd equivariant cohomology groups of $S$ itself vanish.
\begin{remark} \label{tubular neighborhood remark}
The cohomology groups in (C') can be described more geometrically
as follows.  Suppose that there exists a $T$-invariant open neighborhood 
$U = U_R$ of $R$ and a $T$-equivariant deformation retraction $U \times [0,1] \to U$ 
onto $R$ so that for all points $p$ and $t \in [0,1)$, the image of $(p,t)$
is contained in the same stratum as $p$ (the assumption that $\sS$ is a $T$-stratification
implies that such a neighborhood exists locally near every orbit in $R$).
Then there is an isomorphism
\[H^\udot_T(R; j_{R}^*j_{S*}\R_{S,T}) \cong H^\udot_T(U \cap S).\]
\end{remark}


We want to consider objects in the equivariant derived category of 
$\M$ whose cohomology along each stratum 
vanishes in odd degrees.  First consider the case of a single stratum
$S$.  Let $D^{ev}_T(S)$ be the full subcategory of $D^b_T(S)$ consisting of objects
whose cohomology sheaves are locally constant (hence constant by (B)) and 
vanish in odd degrees.

\medskip
\begin{lemma} \label{parity vanishing}
The following statements hold for any object $B$ of $D^{ev}_T(S)$.
\begin{enumerate}
\item[(a)] The equivariant cohomology $H^\udot_T(S; B)$ is a free $H^\udot_T(S)$-module
which vanishes in odd degrees.  For any $x\in S$, restricting cohomology induces
isomorphisms
\[H^\udot_T(S;B) \otimes_{H^\udot_T(S)} H^\udot_T(Tx) \cong H^\udot_T(Tx; B),\; \text{and}\]
\[H^\udot_T(S;B) \otimes_{H^\udot_T(S)}\R \cong H^\udot(\{x\}; B).\]
\item[(b)] For any stratum $R \in \sS$, the natural map
\[H^\udot_T(S; B) \otimes_{H^\udot_T(S)} H^\udot_T(R; j_R^*j^{}_{S*}\R_{S,T}) \to 
H^\udot_T(R; j_R^* j^{}_{S*}B)\]
is an isomorphism (in particular the right hand side vanishes in odd degrees).
\item[(c)] For any $C\in D^b_T(S)$, the natural maps
\[\Hom^\udot_{D^b_T(S)}(B, C) \to \Hom_{H^\udot_T(S)}(H^\udot_T(S;B), H^\udot_T(S;C))\]
and
\[H^\udot_T(S;B \otimes C) \to H^\udot_T(S;B) \otimes_{H^\udot_T(S)} H^\udot_T(S;C)\]
are isomorphisms.
\end{enumerate}
 \end{lemma}

\begin{proof} 
First we prove (a).  
Let $[2l,2m]$ be the largest interval on which $H^\udot(B)$ is supported;
we proceed by induction on the difference $m-l$.
 If $m-l = 0$, then $B$ is
a $T$-equivariant local system placed in degree $2m$.
Our condition (B) says it must be a constant local system, 
in which (a) is obvious.
Otherwise,
consider the distinguished triangle $\tau_{<2m}B \to B \to \tau_{\ge 2m}B$.
The statement holds for $\tau_{<2m}B = \tau_{<2m-1}B$ and $\tau_{\ge 2m}B =
H^{2m}(B)[-2m]$ by the inductive hypothesis.  
Taking equivariant cohomology of this triangle gives a 
long exact sequence; since the extreme terms are free modules vanishing
in odd degrees, the middle is as well.  

The other statements are now proved by a similar induction, using
the five-lemma and the freeness of $H^\udot_T(S;B)$.
\end{proof}

When the space $\M$ has multiple strata, we consider sheaves which satisfy 
parity vanishing along each stratum.

\begin{definition}  Let $B \in D^b_T(\M)$. 
We say that $B$ has ``$*$-parity vanishing'' (respectively
``$!$-parity vanishing'') if  $j^*_S B$ (resp.\ $j^!_S B$) is in $D^{ev}_T(S)$ for all $S \in \sS$.
Let $D^*_T(\M)$ (respectively $D^!_T(\M)$) denote the full subcategory of $D^b_T(\M)$
consisting of such objects.
\end{definition}

Note that $*$-parity vanishing is equivalent to $\sS$-constructibility 
plus the vanishing of the ordinary cohomology sheaves in odd degrees. 
The constant equivariant sheaf $\R_{\M,T}$ has $*$-parity vanishing, 
but does not have $!$-parity vanishing in general unless $\M$ is smooth.  
If $\N$ is a locally closed union of strata of $\M$
and $j\colon \N\to \M$ is the inclusion, then $j^*$ and
$j_!$ preserve $*$-parity vanishing, while $j^!$ and $j_*$ 
preserve $!$-parity vanishing. 

\begin{lemma} \label{shriek and star} If $B\in D^b_T(\M)$ is $\sS$-constructible and
satisfies $!$-parity vanishing, then for any stratum $S$ the equivariant cohomology 
$H^\udot_T(S; j_S^*B)$ vanishes in odd degrees. 
\end{lemma}

\begin{proof}
Choose an ordering $S_1, \dots, S_r$ of the strata so that 
$U_k = \bigcup_{i=1}^k S_i$ is open in $\M$ for all $k=1,\dots,r$, and
let $B_k = f_{k*}f_k^* B$, where $f_k\colon U_k \to \M$ is the inclusion.
In particular $B_0 = 0$ and $B_r = B$.   
Since $j_{S_k}^!B \in D^{ev}_T(S_k)$, Lemma \ref{parity vanishing}(b) implies that 
$H^\udot_T(S;j_S^*j^{}_{S_k*}j_{S_k}^!B)$ vanishes in odd degrees.
The lemma follows by induction on $k$ using the distinguished triangles
$j^{}_{S_k*}j_{S_k}^!B  \to B_k \to B_{k-1}$.  
\end{proof}

For any $T$-space $X$ and objects $B,C \in D^b_T(X)$, 
let $$\Hom^k_X(B,C) = \Hom_{D^b_T(X)}(B,C[k])$$ and
$$\Hom^\udot_X(B,C) = \bigoplus_{k\in \Z} \Hom^k_X(B,C).$$
We want to describe the space of homomorphisms $B \to C$ by looking at 
their restrictions $j^*_SB \to j^*_SC$ to each stratum $S$.  
These restrictions for different strata are constrained by a compatibility 
condition which can be 
described formally using the adjunction
\[\Hom^\udot_S(j_S^*B, j_S^*C) = \Hom^\udot_\M(B, j_{S*}j_S^*C).\]
To simplify notation, define $\Phi_S = j_{S*}j^*_S$.  Then the restriction
of a morphism to $S$ can be rephrased as the map
\begin{equation}\label{morphism localization}
\Hom^\udot_\M(B, C) \to \Hom^\udot_\M(B, \Phi_S C)
\end{equation}
obtained by composing with the adjunction map 
$C \to \Phi_S C$.
The compatibility condition arises because the functoriality of $\Phi_S$
gives rise to a commutative square
\begin{equation}\label{compatibility}
\xymatrix{\Hom^\udot_\M(B, C) \ar[r]\ar[d] & \Hom^\udot_\M(B, \Phi_S C) \ar[d]\\
\Hom^\udot_\M(B, \Phi_R C) \ar[r] & \Hom^\udot_\M(B, \Phi_R \Phi_S C)
}\end{equation}
for any pair of strata $R$, $S$ with $R \subset \ol{S}$.

\begin{theorem}\label{big localization theorem}
Consider a pair of objects $B \in D^*_T(\M)$ and $C\in D^!_T(\M)$.  
The group $\Hom^k_\M(B, C)$ vanishes when $k$ is odd, and
the localization \eqref{morphism localization}
identifies $\Hom^\udot_\M(B,C)$ with the set of tuples 
$$(f_S)\in\bigoplus_{S\in \sS}\Hom^\udot_\M(B, \Phi_S C)$$ 
such that for any strata $R$ and $S$ with
$R \subset \ol S$, $f_R$ and $f_S$ map under \eqref{compatibility} 
to the same element of $\Hom^\udot_\M(B, \Phi_R \Phi_S C)$.
\end{theorem}


\begin{remark}\label{left exact}
We can express the second statement of the theorem in another 
way that will be useful.  Fix $B$, and consider the 
commutative square 
\eqref{compatibility} as a functor of $C$.
Negating one of the arrows in \eqref{compatibility} gives
a functor
\begin{equation} \label{localization exact sequence}
\Hom^\udot_\M(B, -) \to  \bigoplus_S \Hom^\udot_\M(B,\Phi_S(-)) \to 
\mathop{\bigoplus_{R \subset \ol{S}}}_{R \neq S} \Hom^\udot_\M(B,  \Phi_R\Phi_S(-))
\end{equation}
from $D^b_T(\M)$ to the category of three-term complexes
of graded vector spaces.  Theorem \ref{big localization theorem} 
asserts that the evaluation of this functor on $C$ is left-exact.
\end{remark}

\smallskip
\begin{proof}
We use induction on the number of strata in $\sS$.
If $|\sS| = 1$, then the parity vanishing of 
$\Hom^k_\M(B, C)$ follows from Lemma \ref{parity vanishing}(c), while 
the second statement is trivial.

Now suppose $|\sS| > 1$, and assume the theorem is true for all
spaces with fewer strata.  
let $S_0$ be a closed stratum of $\sS$, and let 
$U = \M \setminus S_0$.  Applying 
\eqref{localization exact sequence}
to the distinguished triangle $j^{}_{S_0*}j_{S_0}^!C \to C \to j^{}_{U*}j_U^*C$
gives a long exact sequence of chain complexes.  
In fact, it breaks into short exact sequences of chain
complexes.  For the term $\Hom^\udot_\M(B, -)$,
this follows from the induction hypothesis, since 
by adjunction $\Hom_\M^\udot(B, j^{}_{S_0*}j_{S_0}^!C)$
and $\Hom_\M^\udot(B, j^{}_{U*}j_U^*C)$ can be expressed as homomorphisms on
$S_0$ and $U$ between objects which satisfy the hypotheses of the theorem,
and so they vanish in odd degrees.
For the second term, we have 
$\Hom^\udot_\M(B, \Phi_S C) = 
\Hom^\udot_S(j_S^*B, j_S^*C)$, which vanishes in odd degrees
whenever $C \in D^!_T(\M)$, by Lemma \ref{parity vanishing} and Lemma \ref{shriek and star}.
For the third term, just note that 
$\Phi_R\Phi_S j^{}_{S_0*}j_{S_0}^!C = 0$ whenever $R \neq S$.

The theorem now follows from the snake lemma if we can show that
the chain complexes \eqref{localization exact sequence}
coming from $j^{}_{S_0*}j_{S_0}^!C$ and $j^{}_{U*}j_U^*C$
are left exact.  In both cases this follows from applying the 
inductive hypothesis, again using adjunction to express the first term
in \eqref{localization exact sequence} as 
homomorphisms on $S_0$ and $U$.   Note that the last entry of the 
chain complex for $j^{}_{U*}j_U^*C$ will
contain extra terms of the form $\Hom^\udot_\M(B,  \Phi_{S_0}\Phi_S(C))$
which do not appear in the complex for $j_U^*C$, but they
do not affect exactness on the left.   
\end{proof}

\vspace{-\baselineskip}
\subsection{Localization} \label{refined localization}
We now use Theorem \ref{big localization theorem} to describe some cohomology
and homomorphism groups of $D^{b}_T(\M)$ in terms of modules 
over a sheaf of rings on a finite poset $\wh\sS$.  
Elements of this poset are of two types: (1) strata $S \in \sS$, and
(2) pairs $(R,S)$ of strata with $R \subset {\ol S}$, $R \neq S$.
The partial order $\preceq$ on this set is given by letting 
$(R, S) \preceq R$ and $(R, S) \preceq S$ for all $(R, S)$
of type (2); two different elements which are both of type (1) or of
type (2) are incomparable.  Note that the closure relations among the
strata do not appear in this partial order.  We consider sheaves on 
this poset as described in Section \ref{sheaves on posets}.
We will abuse notation slightly and denote the stalk of 
a sheaf $\cS$ on $\wh\sS$ at $(R,S)$ by $\cS(R,S)$ rather
than $\cS((R,S))$.  We will not need the earlier notation
$\cS(U,V) = \ker(\cS(U) \to \cS(V))$, so this should not
cause confusion.

Define a localization functor $\Loch$ from 
$D^b_T(\M)$ to sheaves of $H_T^\udot(pt)$-modules 
on $\wh\sS$ by defining its stalks to be
$$(\Loch B)(S) := H^\udot_T(S;B) = H^\udot_T(\M; \Phi_S B)$$ and
$$(\Loch B)(R,S) := H^\udot_T(\M; \Phi_R\Phi_S B) = H^\udot_T(R; j_R^*j_{S*}j_S^*B)$$
and letting the maps between the stalks be induced by the 
natural maps $\Phi_R B \to \Phi_R\Phi_S B$ 
and $\Phi_S B \to \Phi_R\Phi_S B$. 
Note that if there exists a $T$-invariant 
tubular neighborhood of $R$ as in Remark \ref{tubular neighborhood remark},
then $(\Loch B)(R,S)$ can be described more geometrically as 
$H^\udot_T(U_R\cap S; B)$.

Let $\ulA$ denote the constant sheaf on $\wh{\sS}$ with stalk 
the polynomial ring $A = H^\udot_T(pt)$.

\begin{lemma}\label{tensor} For any $B,C \in D^b_T(\M)$, there is a natural 
morphism
\begin{equation}\label{tensor map}
\phi_{B,C}\colon \Loch B \otimes_{\ulA} \Loch C \to \Loch (B\otimes C).
\end{equation}
If $D$ is another object in $D^b_T(\M)$, then the two maps
\[
\phi_{B\otimes C, D} \circ (\phi_{B,C} \otimes \id_D),
\phi_{B, C\otimes D} \circ (\id_B\otimes \phi_{C,D})\colon \Loch B \otimes_{\ulA} \Loch C \otimes_{\ulA} \Loch D \to \Loch(B \otimes C \otimes D)
\]
are equal.
If $B = \R_{\M,T}$, the resulting maps 
\[
H^\udot_T(S) \otimes_{H_T^\udot(pt)} H^\udot_T(S;C)  \to H^\udot_T(S; C), \,\,\,\,S \in \sS 
\]
give the usual action of cohomology.
\end{lemma}
\begin{proof} 
A more precise statement of the associativity constraint on $\phi$ 
should include the natural isomorphisms
between the different ways of associating the triple tensor products.  
This, together with the existence of a morphism $\ulA \to \Loch \R_{\M,T}$
compatible with the isomorphisms $\ulA \otimes_\ulA \cB \cong \cB \cong \cB \otimes_\ulA \ulA$ \,and\,
$\R_{\M,T} \otimes B \cong B \cong B \otimes \R_{\M,T}$, is 
the definition of the statement
that $\Loch$ is (or forms part of) a modular functor from
$D^b_T(\M)$ to $\ulA$-modules.  The functors $j_{S*}$, $j_S^*$, 
and $H^\udot_T$ are all modular functors, which implies that
$\Loch$ is also.
\end{proof}

In particular this means that the sheaf 
$\wh\cA = \Loch \R_{\M,T}$ is a sheaf of graded rings on $\wh\sS$
with multiplication given by the cup product, and
for any $B \in D^b_T(\M)$ the sheaf $\Loch B$ is naturally a
(left) graded $\wh \cA$-module.  
Our assumptions on parity vanishing imply
that $\wh\cA$ is commutative and the left and right module structures 
on $\Loch B$ coincide.

Applying the associativity
constraint on $\phi$ to 
\[B \otimes \R_{\M,T} \otimes C \cong B \otimes C\]
(note that the two ways of constructing this isomorphism
are equal), we conclude that $\phi_{B,C}$ descends to a 
naural map
\begin{equation}\label{tensor over A}
 \Loch B \otimes_{\wh\cA} \Loch C \to \Loch (B\otimes C).
\end{equation}
(Again the more precise statement is that
$\Loch$ is a modular functor from $D^b_T(\M)$ to
$\wh\cA-\mathop{mod}$.)
Lemma \ref{parity vanishing} implies that this map
is an isomorphism if both $B$ and $C$ are in 
$D^*_T(\M)$: the second part of statement (c) 
implies that it is an isomorphism at points of $\sS$, 
and (b) implies that it is an isomorphism at points
$(R, S) \in \wh\sS \smallsetminus \sS$.
Theorem \ref{big localization theorem} can now be reformulated as 
follows.

\begin{theorem} \label{localized sheaves}
If $B \in D^*_T(\M)$, $C \in D^!_T(\M)$, then the natural 
map 
\[\Hom^\udot_\M(B,C) \to \Hom^\udot_{\wh\cA-\mathop{mod}}(\Loch B, \Loch C)\]
is an isomorphism.
In particular, taking $B = \R_{\M,T}$, the global sections of $\Loch C$ are
canonically isomorphic to $H^\udot_T(\M; C)$.
\end{theorem}
\begin{proof}
Let $\cB = \Loch B$ and $\cC = \Loch C$.  A map $\cB \to \cC$ consists of
maps $\cB(S) \to \cC(S)$ and $\cB(R,S) \to \cC(R,S)$ over all $S$ and $(R,S)$ in
$\wh\sS$, compatible with the restriction maps in the sheaves $\cB$ and $\cC$.
Because $B \in D^*_T(\M)$, Lemma \ref{parity vanishing}(b) implies that the 
map $\cB(S) \otimes_{\wh\cA(S)} \wh\cA(R,S) \to \cB(R,S)$ is an isomorphism.
Thus the maps $\cB(R,S) \to \cC(R,S)$ are determined by 
the maps $\cB(S) \to \cC(S)$ for $S \in \sS$.  These maps will determine a 
map of sheaves if and only if for each $(R,S) \in \wh\sS$, the composition
\[\cB(R) \to \cC(R) \to \cC(R,S)\]
is equal to 
\[\cB(R) \to \cB(R,S) \stackrel{\sim}{\longleftarrow} \cB(S) \otimes_{\wh\cA(S)} \wh\cA(R,S)
\to  \cC(S) \otimes_{\wh\cA(S)} \wh\cA(R,S) \to \cC(R,S).\]
But by Lemma \ref{parity vanishing}(c) we have isomorphisms
\[\Hom^\udot_{\wh\cA(S)}(\cB(S), \cC(S)) \cong \Hom^\udot_S(B, \Phi_S(C))\]
and
\[\Hom^\udot_{\wh\cA(R)}(\cB(R), \cC(R, S)) \cong \Hom^\udot_R(B, \Phi_R\Phi_S(C)),\]
and with these identifications Theorem \ref{big localization theorem}
shows that the space of homomorphisms $\cB\to \cC$
is isomorphic to $\Hom^\udot_\M(B, C)$.
\end{proof}

\vspace{-\baselineskip}
\subsection{Generalized minimal extension sheaves}\label{GMES} 
We want to use our functor $\Loch$ to describe intersection cohomology. 
We need to assume that $\IC_{\M,T} \in D^*_T(\M)$, or equivalently
that the cohomology sheaves of $\IC_\M$ vanish in odd degrees.  
Since $\IC_{\M,T}$ is Verdier self-dual, this 
also implies that $\IC_{\M,T} \in D^!_T(\M)$.
This condition holds for many interesting singular spaces, including toric
varieties, Schubert varieties, and hypertoric varieties; some
general theorems which imply it can be found in \cite{BriJ}.
(Note, however, that many varieties covered by these results
will \emph{not} have a stratification satisfying our condition 
(C').)

We now define a class of $\wh\cA$-modules called 
{\em generalized minimal extension sheaves}
(GMES) in analogy with the minimal extension sheaves of 
\cites{BBFK2, BreLu} described in Section \ref{pure sheaves}.  
The definition is also very similar to the indecomposable pure
sheaves on moment graphs which were defined in  
\cite{BM} to compute torus-equivariant 
intersection cohomology of Schubert varieties.

For any stratum $S \in \sS$, let 
$\bdy S = \{R \in \sS \mid S \subset \ol{R}, R \neq S\}$.


\begin{definition}\label{GMES definition}
We say an $\wh\cA$-module
$\wh\cL$ is a GMES if the following conditions hold:
\begin{enumerate}
\item $\wh\cL(S_0) \cong \wh\cA(S_0)$ for the open stratum $S_0$.
\item For each $S \in \sS$, $\wh\cL(S)$ is a free $\wh\cA(S)$-module.
\item For each $(R, S) \in \wh\sS$, we have 
\[\wh\cL(R,S) \cong \wh\cL(S)\otimes_{\wh\cA(S)} \wh\cA(R,S),\]
and the map $\wh\cL(S) \to \wh\cL(R,S)$ is the natural one 
coming from the tensor product.
\item For each $S \in \sS$ the restriction maps from 
$\wh\cL(S)$ and $\wh\cL(p^{-1}(\bdy S))$ to 
$\bigoplus_{R \in \bdy S} \wh\cL(S,R)$ have the same image.
\item $\wh\cL$ is minimal with respect to conditions 1-4.
\end{enumerate}
\end{definition}

\begin{remark}
Condition 4 is equivalent to saying that $\wh\cL$ is:
(1) generated by global sections, i.e.\ the 
restriction $\wh\cL(\wh{\sS}\,) \to \wh\cL(S)$ is surjective
for every $S \in \sS$, and (2) flabby for the 
coarser topology on $\wh\sS$ whose open sets are
$p^{-1}(U)$ for $U\subset \sS$ open in the order 
topology, where the map $p\colon \wh\sS \to \sS$
is defined by $p(S) = S$ and $p((R,S)) = S$.
\end{remark}

\begin{proposition}\label{GMES is unique}
Any two generalized minimal extension sheaves are 
isomorphic.
\end{proposition}
The proof of this result is essentially the same as for
ordinary minimal extension sheaves; see \cite[Theorem 2.3]{BBFK2} and \cite[Theorem 5.3]{BreLu}.

\begin{theorem} \label{GMES theorem}
Suppose that $\IC_{\M,T} \in D^*_T(\M)$ and that 
\[ \tag{*} \text{for every $(R, S) \in \wh\sS$, the map
$\wh\cA(S) \to \wh\cA(R,S)$ is surjective.}\]  
Then
$\wh \cL := \Loch \IC_{\M,T}$ is a generalized minimal extension
sheaf.  Furthermore it is rigid, meaning that it has only scalar automorphisms as a 
graded $\wh\cA$-module.
\end{theorem}
 
\begin{proof}  Put $\IC = \IC_{\M,T}$.
Condition 1 of Definition \ref{GMES definition} 
is immediate, since $\IC|_{S_0} = \R_{S_0,T}$.
Conditions 2 and 3 follow from Lemma \ref{parity vanishing}.
To prove conditions 4 and 5, take any
$S \in \sS$, and let $U = \bigcup_{R \in\bdy S} R$.  We need to show that
$\wh\cL(S) = H^\udot_T(S;j_S^*\IC)$
is the minimal free $\wh\cA(S)$-module which surjects onto
the image of 
\begin{equation}\label{restriction to boundary}
\wh\cL(p^{-1}(\bdy S)) \longrightarrow
\bigoplus_{R \in\bdy S} \wh\cL(S,R).
\end{equation}

The long exact sequence
\[ \dots \to H^\udot_T(S;j_S^!\IC) \to H^\udot_T(S; j_S^*\IC) \to H^\udot_T(S;
j^*_Sj^{}_{U*}j^*_U\IC) \to \dots\]
breaks into short exact sequences,  
since the first and third terms vanish in odd degrees by Lemmas \ref{parity vanishing} 
and \ref{shriek and star}.  Combining  
Lemma \ref{parity vanishing}(a) with the degree restrictions on stalks and
costalks of $\IC$, we see that 
$H^\udot_T(S; j_S^*\IC)$ is generated in degrees less than
the complex codimension of $S$, while $H^\udot_T(S;j_S^!\IC)$ is generated in degrees
greater than that.  This implies that $\wh\cL(S)$ is the smallest
free $\wh\cA_S$-module which surjects onto 
$H^\udot_T(S; j^*_Sj^{}_{U*}j^*_U\IC)$.

Since $\wh\cL(p^{-1}(\bdy S)) \cong \IH^\udot_T(U)$ by Theorem \ref{localized sheaves}, 
the map \eqref{restriction to boundary} can be factored as
\[\IH_T^\udot(U) \to H_T^\udot(S;j^*_Sj^{}_{U*}j^*_U\IC)\to
\bigoplus_{R\in \bdy S}  
H^\udot_T(S; j_S^*\Phi_R\IC) =  \bigoplus_{R \in\bdy S} \wh\cL(S,R),\]
so our result will follow if we can show that the first map 
is surjective and the second map is injective.
The surjectivity follows from the more general fact that 
$H^\udot_T(U; B) \to H^\udot_T(S;j^*_S B)$ is surjective for any
$B \in D^!_T(U)$, which can be proved by an induction similar
to the proof of Lemma \ref{shriek and star}; the fact that 
the statement holds for $j^{}_{R*}j_R^!B$ if $R \in \bdy S$
follows from the property (*) and 
Lemma \ref{parity vanishing}(b).  
Injectivity of the second map also follows from a more general statement, that 
$H^\udot_T(S;j^*_Sj^{}_{U*}j^*_U B) \to \bigoplus_{R \in\bdy S} H^\udot(\M; \Phi_S\Phi_RB)$
is injective for any object $B \in D^!_T(\M)$.  This again can be proved along the
lines of Lemma \ref{shriek and star}.  

The fact that the sheaf $\wh \cL$ has only scalar automorphisms follows 
by induction on the number of strata, using the degree constraints
on the generators of $\wh\cL_S$ and $H^\udot_T(S; j_S^!\IC)$.
\end{proof}

\vspace{-\baselineskip}
\begin{remark} \label{costalk remark}
It also follows from our  proof that 
the equivariant intersection cohomology 
$\IH^\udot_{T,S}(\M) = H^\udot_T(S;j^!_S\IC)$
with supports along a stratum $S$
 is isomorphic to the kernel of the map
$\wh\cL(S) \to \bigoplus_{\bdy S} \wh\cL(S,R)$.
\end{remark}

\begin{remark}
One example which satisfies all our hypotheses is when $\M = X_\Sigma$
is the toric variety defined by a fan $\Sigma$, and $T$ is the
maximal compact subgroup of the natural complex torus
$T_\C$.  In that case, we recover the theory of 
minimal extension sheaves defined in \cites{BBFK2,BreLu}.
The strata $S_\sigma$ are 
just the $T_\C$-orbits indexed by cones $\sigma \in \Sigma$. 
For any pair of strata $S_\sigma$, 
$S_\tau$ with $S_\tau \subset S_\sigma$, there is an equivariant
deformation retraction of $S_\sigma$ to its intersection with a 
tubular neighborhood of $S_\tau$,
which implies that the sheaf $\wh \cA$ is the pullback of the 
structure sheaf $\cA$ of conewise polynomial functions on
$\Sigma$ via the continuous map $p\colon\wh \Sigma \to \Sigma$,
$\sigma \mapsto \sigma$, $(\tau, \sig) \mapsto \sig$, and
the generalized minimal extension sheaf $\wh \cL$ is the
pullback of the toric minimal extension sheaf $\cL$.
\end{remark}

\begin{remark}
If the property (*) does not hold, it 
is still possible to calculate the sheaf $\Loch \IC$,
but definition \ref{GMES definition} must be modified.
The problem is that although $\wh\cL(S)$ surjects onto
$H^\udot_T(S; j^*_Sj^{}_{U*}j^*_U\IC)$, this module may not be the
same as the image of the map
$\IH^\udot_T(U) \to \bigoplus_{R > S} \wh\cL_{(S,R)}$.
Instead, it must be computed as a submodule of 
$\bigoplus_{R > S} \wh\cL_{(S,R)}$ cut out
by compatibility relations coming from maps to 
$H^\udot_T(S; \Phi_S\Phi_R\Phi_{R'}\IC)$.
For an example of a space where (*) fails, consider the nilpotent cone
in $\mathfrak{sl}_3(\C)$ with the stratification by adjoint orbits,
and the action of the maximal torus $T$ in $SL_3(\C)$.
\end{remark} 

We now use our results to study maps $\ICxIC \to \IC$.
Suppose that such a map
makes $\IC$ into a commutative ring object, as in Theorem \ref{IC is ring object}.
Then Lemma \ref{tensor} implies that $\wh\cL = \Loch \IC$ is a sheaf of commutative
$\wh\cA$-algebras.  With our hypotheses, we will show the converse.

\begin{theorem}\label{lift}
Suppose that our stratified $T$-space $(\M,\sS)$ satisfies properties (A), (B) and (C') of
Section \ref{stratum-by-stratum} and property (*) of Theorem \ref{GMES theorem}, and that $\IC \in D^*_T(\M)$.
Then any commutative $\wh\cA$-algebra structure on $\wh\cL = \Loch \IC$ lifts uniquely
to a commutative ring structure on $\IC$.
\end{theorem}

\begin{proof}
Since $\IC$ is self-dual, it also
lies in $D^!_T(\M)$.  Since the tensor product commutes with 
taking stalks, we have $\ICxIC \in D^*_T(\M)$.  
Applying \eqref{tensor over A} gives a 
natural isomorphism $\wh\cL
\otimes_{\wh\cA} \wh\cL\cong \Loch(\ICxIC)$.  
Thus Theorem \ref{localized sheaves} gives an isomorphism
\[\Hom^\udot_{D^b_T(\M)}(\ICxIC,\IC) \to  
\Hom^\udot_{\wh\cA\md}(\wh\cL \otimes_{\wh\cA}\wh\cL, \wh\cL).\]

Let $m\colon \ICxIC \to \IC$ be
the morphism corresponding to the multiplication map 
of $\wh\cL$.  Associativity of this product is expressed
as the equality of two maps $\ICxIC \otimes \IC \to \IC$, 
so another application of Theorem \ref{localized sheaves}
allows us to deduce associativity for $m$ from associativity
of the multiplication on $\wh\cL$.  
The compatibility of the unit map 
$\R_{\M,T} \to \IC$ with $m$ follows in a similar way, since 
its localization gives the natural map $\wh\cA \to \wh\cL$,
and the isomorphism $\R_{\M,T} \otimes \IC \to \IC$ localizes
to give the natural isomorphism 
$\wh\cA \otimes_{\wh\cA} \wh\cL \to \wh\cL$
(see Lemma \ref{tensor}).
\end{proof}

\vspace{-\baselineskip}
\begin{remark}\label{enough}
Note that because $\wh\cL$ is isomorphic to any generalized minimal
extension sheaf by a unique isomorphism, to construct a commutative ring structure
on $\IC$ it is enough to produce any sheaf of commutative $\wh\cA$-algebras
which is a GMES.
\end{remark}

\subsection{The structure sheaf for hypertoric varieties} \label{hypertoric structure sheaf}
Let $\cH$ be a central, unimodular arrangement in the vector space $V$, with 
hyperplanes indexed by a finite set $I$, and consider the hypertoric variety $\M_\cH$
along with the stratification described in Section \ref{Hypertoric basics},
in which the strata $S_F$ are indexed by flats $F \in L = L_\cH$.  
To simplify notation, from now on
we will identify the stratum with the flat that names it in our notation.
Thus our extended poset $\wh L$ will consist of
single flats and pairs $(E,F)$ of flats with
$E < F$, rather than strata and pairs of strata.
In this section we give an explicit description of the structure sheaf $\wh\cA$ of $\wh L$
and use it to conclude that $\M_\cH$ satisfies all of
the hypotheses of Theorem \ref{lift}, thus reducing Theorem \ref{IC is ring object}
to a statement about the minimal extension sheaf $\wh\cL$ on $\wh L$.

We begin by fixing some notation.
For any subset
$I' \subset I$, let $\R[I']$ denote the polynomial ring with generators
$\be_i$ for $i\in I'$, and let $Q_{I'}$ be the ideal in this ring generated
by $\be_i^2$, $i\in I'$.  Let $J_\cH \subset \R[I]$ be the ideal
$$\left< \sum_{i\in I_0}a_ie^{}_{I_0 \smallsetminus \{i\}} 
\bigmid I_0 \subset I\;\, \text{and}\;\,\sum_{i\in I_0} a_i x_i|_V = 0\right>.$$
Thus the ring $R(\cH)$ introduced in Section \ref{defining R} is equal to $\R[I]/J_\cH$.
For each flat $F$, we get ideals $J_{\cH^F}$ and $J_{\cH_F}$
in the rings $\R[I\setminus F]$ and $\R[F]$, respectively.  We will
abuse notation and use the same symbols $Q_{I'}$, $J_{\cH^F}$, 
$J_{\cH_F}$ to refer to the
ideals that they generate in $\R[I]$, using the obvious inclusions
$\R[I'] \subset \R[I]$ for any $I' \subset I$.

\begin{lemma}\label{ideal identity}
For any flat $f \in L_\cH$, we have the following equality of ideals in $\R[I]$:
\[J_\cH + J_{\cH^F} + Q_{I\smallsetminus F} = J_{\cH_F} + J_{\cH^F} + Q_{I\smallsetminus F}.\]
\end{lemma}
\begin{proof} One inclusion is clear, since the generators of $J_{\cH_F}$
form a subset of the generators of $J_{\cH}$.  For the other direction,
consider a generator $f=\sum_{i\in I_0} a_i \be_{I_0 \smallsetminus \{i\}}$
of $J_\cH$ 
with $a_i\neq 0$ for all $i\in I_0$.
Partition $I_0$ into $I_1 = I_0 \cap F$ and $I_2 = I_0 \setminus I_1$.
If $I_2 = \emptyset$ then $f \in J_{\cH_F}$.
Otherwise, let $$g = \sum_{i\in I_1} a_i\be^{}_{I_1 \smallsetminus \{i\}}\,\,\,\text{and}\,\,\, 
h = \sum_{j\in I_2} a_j\be^{}_{I_2 \smallsetminus \{j\}},$$ so that
$$f = \be^{}_{I_2}g + \be^{}_{I_1}h.$$
Then $h \in J_{\cH^F}$, and for any $j\in I_2$, we have
$$a_j\be^{}_{I_2} = e_j h \,\,\,- \sum_{k\in I_2\smallsetminus\{j\}} a_k e_j \be^{}_{I_2\smallsetminus\{k\}} 
\,\,\in\,\, J_{\cH^F} + Q_{I\smallsetminus F}.$$
Dividing by $a_j$, we conclude that $e^{}_{I_2}\in J_{\cH^F} + Q_{I\smallsetminus F}$, and therefore
$f \in J_{\cH^F} + Q_{I\smallsetminus F}$.
This completes the proof.
\end{proof}

Recall that for any flat $F \in L_\cH$, the stabilizer of any point of $S_F$
is the subtorus $T_F\subset T$ with Lie algebra $\la F\ra^\bot \subset V^*$.
Then $T/T_F$ is the natural torus that acts on the closure
$\overline{S_F} \cong \M_{\cH^F}$.  
For this reason, we may also think of $T/T_F$ as a quotient of the coordinate torus 
$T^{I\smallsetminus F}$.  Let $$A^F = \Sym\la F \ra = H_{T/T_F}^\udot(pt).$$
Then we have natural inclusions $$A^F \subset A = \Sym V\,\,\,\text{and}\,\,\, 
A^F \subset \R[I\setminus F].$$

Let $\wh\cA$ be the structure sheaf on $\wh L$ induced by
the hypertoric variety $\M_\cH$, as described in Section \ref{refined localization}.
The following proposition gives an explicit description of this
sheaf.   Let
\[S(\cH) = \R[I]/(J_{\cH}+Q_{I}).\]

\begin{proposition}\label{generalized structure sheaf}
 There are canonical isomorphisms
\[\wh\cA(F) \cong A \otimes_{A^F} S(\cH^F)\] and
\[\wh\cA(E,F) \cong A \otimes_{A^F} S(\cH^F) \otimes_{\R[I\smallsetminus E]} S(\cH^E) = 
A \otimes_{A^F} \R[I\setminus F]/(J_{\cH^F} + J_{\cH^E}+Q_{I\smallsetminus F}),\]
and the restriction maps $\wh\cA(F) \to \wh\cA(E,F)$ and $\wh\cA(E) \to \wh\cA(E,F)$
are the obvious ones.
\end{proposition}

\begin{proof} First, note that tensoring with $A$ over $A^F$ is just
the change of coefficients that gives $T$-equivariant cohomology from
$T(F)$-equivariant cohomology, so the general case follows from the case 
where $F = \emptyset$ is the minimal flat.

We obtain the identification of $\wh\cA(\emptyset) = H^\udot_T(S_\emptyset)$
with $S(\cH)$ by noting that $T$ acts freely on $S_\emptyset$,
so there is an isomorphism of rings $H^\udot_T(S_\emptyset) \cong H^\udot(S_\emptyset/T)$.
The moment map $\mu$ identifies $S_\emptyset/T$ 
with the complement $$M(\cH) := \hkv \setminus \bigcup_{i\in I} \hkhi$$ of the 
tripled arrangement.
By \cite[5.6]{dLS}
this ring is naturally isomorphic to the quotient
$\R[I]/(J_\cH+Q_I)$.\footnote{The generators of $Q_I$ were inadvertently
omitted from the statement of \cite[5.6]{dLS}.}  Note that
the formula in \cite{dLS} has $\sum_{i\in I_0} \sgn(a_i) \be_{I_0\smallsetminus\{i\}}$
in place of the generator $\sum_{i\in I_0}a_i \be_{I_0\smallsetminus\{i\}}$ of
$J_\cH$, but since we assume $\cH$ is unimodular, these are the same.

In order to pin down the restriction maps in the structure sheaf we need
to describe this isomorphism more precisely.  In \cite{dLS} the generator
$\be_i$ corresponding to the $i^\text{th}$ hyperplane is the pullback
of a generating class in the degree two cohomology of 
$M(\cH_{\{i\}}) \cong \R^3 \setminus \{0\}$
via the natural quotient map.  This class can be given another way,
using the construction of $\M_\cH = \Psi^{-1}(\hkv)/T_V$ as a hyperk\"ahler quotient of 
$\H^I$.  The open stratum $S_\emptyset$ is a quotient of
$\Psi^{-1}(\hkv)\cap (\H^\times)^I$ by the action of the subtorus $T_V\subset T^I$, 
which acts freely.  Thus we have a ring isomorphism
$\wh\cA(\emptyset) \cong H^\udot_{T^I}(\Psi^{-1}(\hkv)\cap (\H^\times)^I)$, giving rise to
a natural homomorphism $\R[I] = H^\udot_{T^I}(pt) \to \wh\cA(\emptyset)$.  

We claim that this homomorphism is the obvious one which 
sends $\be_i$ to $\be_i$.
This is because the projection 
$S_\emptyset/T \cong M(\cH) \to M(\cH_{\{i\}})$
can be covered by the $T^I$-equivariant 
projection $\Psi^{-1}(\hkv)\cap (\H^\times)^I \to \H^\times$
onto the $i^\text{th}$ factor.
An easy computation shows that 
$H^\udot_{T^I}(\H^\times) \cong \R[I]/\la e^2_i\ra$, and the image
of $\be_i$ gives a generator of the degree $2$ cohomology of
$\H^\times/T^I = \H^\times/U(1) \cong \R^3 \setminus \{0\}$.

Next, consider the ring $\wh\cA(E,\emptyset)$.  It 
is the target of maps from the sources $S(\cH)=\wh\cA(\emptyset)$ and
$S(\cH^E)  \subset \wh\cA(E)$, and by our identification
of the generators $\be_i$, these maps induce a map
\begin{equation}\label{Tubular neighborhood}
S(\cH^E) \otimes_{\R[I\setminus E]} S(\cH)
\to \wh\cA(E,\emptyset).
\end{equation}
We show that this map is an isomorphism. 
By Lemma \ref{ideal identity}, the ring on the left is 
\[\R[I]/(J_{\cH}+J_{\cH^E}+Q_I) = 
\R[I]/(J_{\cH_E}+J_{\cH^E}+Q_I) \cong S(\cH^E)\otimes_\R S(\cH_E).\]
To understand the right hand side of \eqref{Tubular neighborhood},
choose a splitting $T = T_E \times T^E$.  Then 
$T^E$ acts freely on $\M_{\ge E} := \bigcup_{F \le E} S_F$,
so there is an equivalence of categories 
$D^b_T(\M_{\ge E}) \simeq D^b_{T_E}(\M_{\ge E}/T^E)$ which preserves 
cohomology and commutes with pullbacks and pushforwards 
(see \cite[2.6.2 and 3.4.1]{BerLu}).  
If we let $\jhat_F\colon S_F/T^E \to \M_{\ge E}/T^E$
denote the inclusion, then we see that
$\wh\cA(E,\emptyset)$ is isomorphic to the $T_E$-equivariant
cohomology of 
$B = \jhat_E^*(\jhat_\emptyset)_*\R_{S_\emptyset/T^E, T_E}$.

The group $T_E$ acts trivially on $S_E/T^E$, and using Proposition 
\ref{strata and slices} one can show that $B$
is a locally constant sheaf with $T_E$-equivariant stalk cohomology
which is isomorphic to $S(\cH_E)$.  On the other hand,
$S_E/T^E \cong M(\cH^E)$, so we have a spectral sequence
with $E_2$ term $S(\cH^E)\otimes_\R S(\cH_E)$ converging to 
$\wh\cA(E,\emptyset)$.  It collapses for 
parity reasons, so both sides of \eqref{Tubular neighborhood}
have the same graded dimension.

Thus we only have to show that \eqref{Tubular neighborhood}
is a surjection.  This follows by the Leray-Hirsch theorem,
since the image of $\be_i$, $i\notin F$ give 
classes generating the action of the cohomology of the base $M(\cH^E)$,
while the image of $S(\cH_E)$ gives classes which restrict to
a vector space basis for $H^\udot_{T_E}(B|_x)$ for a point $x\in S_E/T^E$.
\end{proof}

\vspace{-\baselineskip}
\begin{corollary}
The hypertoric variety $\M_\cH$ satisfies the hypotheses of Theorem \ref{lift},
thus commutative ring structures on $\IC$ correspond precisely to commutative
$\wh\cA$-algebra structures on $\wh\cL$.
\end{corollary}

\begin{remark}
We believe 
that the ring structure on intersection cohomology lifts to the derived category for
arbitrary unimodular arrangements; more precisely, that Theorem \ref{IC is ring object}
holds even in the non-central case.  However, in the proof of Proposition
\ref{generalized structure sheaf}
we needed to assume the arrangement is central in order to apply the 
results of \cite{dLS}.  For non-central arrangements it is still possible
to compute the ring structure on the cohomology $H^\udot(M(\cH))$, using results
of Deligne, Goresky, and MacPherson \cite{DGM}, but it is more difficult
to write down an explicit presentation for this ring, so more care will be needed.
\end{remark}

\subsection{Generalized minimal extension sheaves for hypertoric varieties}
\label{hypertoric GMES}
In this section we show that $\wh\cL = \Loch\IC$ admits a 
unique commutative $\wh\cA$-algebra structure.
By Remark \ref{enough}, in order 
to define this algebra structure
it is enough to define 
a sheaf $\wh\cR$ of commutative $\wh\cA$-algebras 
and prove that it is a GMES.

Recall the ring $R(\cH) = \R[I]/J_\cH$.
Define the stalks of $\wh\cR$ by 
\[\wh\cR(F) := R(\cH) \otimes_{\R[I\smallsetminus F]} S(\cH^F) \cong 
\R[I]/(J_\cH + J_{\cH^F} + Q_{I\smallsetminus F}),\]
and 
\[\wh\cR(E,F) := R(\cH) \otimes_{\R[I\smallsetminus F]} S(\cH^F)
\otimes_{\R[I\smallsetminus E]} S(\cH^E) \cong
 \R[I]/(J_\cH + J_{\cH^F} + J_{\cH^E}+ Q_{I\smallsetminus F}),\]
and let the restriction maps be the obvious ones.
It is clear that this makes $\wh\cR$ into a sheaf of commutative $\wh\cA$-algebras.

Not surprisingly, there is a close relation between $\wh\cR$
and the sheaf $\cR$ which was defined in Section \ref{defining R}.
Using Lemma \ref{ideal identity}, we have different expressions
for the stalks of $\wh\cR$:
\begin{equation} \label{stalk expressions}
 \wh\cR(F) = R(\cH_F) \otimes_\R S(\cH^F), \;\;\; \wh\cR(E,F) = R(\cH_F) \otimes_\R S(\cH^F_E) \otimes_\R S(\cH^E). 
\end{equation}
Let $\cI \subset \wh\cA$ be the ideal sheaf 
which on a flat $F$ or a pair $(E,F)$ is generated
by $\be_i$, $i\notin F$, and consider the quotients 
$\wh\cA/\cI$ and 
$\wh\cR/\wh\cR\cI = \wh\cR \otimes_{\wh\cA} (\wh\cA/\cI)$.
Their stalks are given by 
\[(\wh\cA/\cI)(F) = (\wh\cA/\cI)(E,F) = A \otimes_{A^F} \R = \cA(F)\]
and 
\[(\wh\cR/\wh\cR\cI)(F) = (\wh\cR/\wh\cR\cI)(E,F) = R(\cH_F) = \cR(F),\]
and in fact we have $\wh\cA/\cI = p^{-1}(\cA)$ and 
$\wh\cR/\wh\cR\cI \cong p^{-1}(\cR)$, where as before
$p$ is the projection $p(F) = F$, $p((E, F)) = F$.
Since the stalks of $p$ are acyclic we have $p_*p^{-1}\cR \cong \cR$, 
giving an isomorphism
\begin{equation}
 \label{relating R and Rhat}
p_*(\wh\cR \otimes_{\wh\cA}\, p^{-1}(\cA)) \cong \cR.
\end{equation}

\begin{example} Let $\cH$ be the central arrangement with three hyperplanes
in $V \cong \R^2$, indexed by $I = \{1,2,3\}$ and
defined by covectors $w_i \in V^*$ satisfying the relation $w_1+w_2+w_3 = 0$.
We have $\wh\cR = \wh\cA$ on all elements of $\wh L$ except the
maximal flat $I$.  Therefore the stalks at these points are given by 
\begin{itemize}
\item $\wh\cR(\emptyset) = \R[I]/\la \be_1\be_2 + \be_2\be_3 + \be_1\be_3, \be_1^2, \be_2^2, \be_3^2\ra$
\item $\wh\cR(\{1\},\emptyset) = \R[I]/\la \be_2+\be_3, \be_1^2,\be_2^2, \be_3^2\ra$, 
\item $\wh\cR(\{1\})=\wh\cR(I,\{1\}) = \R[I]/\la \be_2+\be_3, \be_2^2, \be_3^2\ra$,
\end{itemize}
and similarly with $\{1\}$ replaced by $\{2\}$ and $\{3\}$.
The space of sections of this sheaf on $\wh L \setminus \{I\}$
injects into $\bigoplus_{i\in I} \wh\cR(I,\{i\})$, with image isomorphic to
$$\R[I]\big/\la \be_1\be_2\be_3, \be_1\be_2+\be_2\be_3+\be_1\be_3\ra.$$  Then
$R(\cH) = \R[I]\big/\la \be_1\be_2+\be_2\be_3+\be_1\be_3\ra$
is the smallest free $\R[\be_1 - \be_2, \be_1 - \be_3]$-module that surjects onto
this image.\footnote{We note that this is almost the same as the ring of global sections
in Example \ref{asymmetric}.  This is because the two arrangements are almost the
same; they differ only in the coorientation of the third hyperplane, which is responsible
for the sign that appears in Example \ref{asymmetric}.}
\end{example}

\begin{proposition}\label{R is a GMES}
$\wh\cR$ is a GMES.
\end{proposition}

\begin{proof}
It is clear that the sheaf $\wh\cR$ satisfies
Conditions 1 and 3 of Definition \ref{GMES definition}.  
Since $\wh\cR(F) = R(\cH_F) \otimes_\R S(\cH^F)$, it
is a free $\wh \cA(F) = \cA(F) \otimes_\R S(\cH^F)$-module.

Next we prove conditions 4 and 5.  For a $\wh\cA$-module $\wh\cF$
and a flat $F$, define
\[\wh\cF_{\bdy F} := \Im\left(\wh\cF(p^{-1}(\bdy F)) \to \bigoplus_{E < F} \wh\cF(F, E)\right).\]
Then $\wh\cR(F)$ maps into $\wh\cR_{\bdy F}$, because $\wh\cR$ is a quotient of
the constant sheaf on $\wh L$ with stalk $\R[I]$, so any element of 
$\wh\cF(F)$ extends to a global section of $\wh\cF$.

Take a flat $F \ne \emptyset$, and let
$\wh\cL = \Loch\IC$.  Consider 
the exact sequences
\begin{equation}\label{exact sequence 1}
 0 \to \Ker_{\wh\cR} \to \wh\cR(F) \to \wh\cR_{\bdy F}
\end{equation}
and 
\begin{equation}\label{exact sequence 2}
0 \to \Ker_{\wh\cL} \to \wh\cL(F) \to \wh\cL_{\bdy F} \to 0.
\end{equation}
The second sequence is right exact because $\wh\cL$ is a GMES, and
so satisfies condition $4$.  We need to show that the first sequence is also 
right exact. 

We can assume by induction that $\wh\cR$ satisfies conditions 4 and 5 for 
all flats $E < F$.  This means that $\wh\cR$ restricts to a GMES on $\wh{\bdy F}$, 
and so $\wh\cR|_{\wh{\bdy F}} \cong \wh\cL|_{\wh{\bdy F}}$, 
by Proposition \ref{GMES is unique}.  Since both
$\wh\cR$ and $\wh\cL$ satisfy condition 3, this isomorphism
extends to $p^{-1}(\bdy F)$, and so $\wh\cR_{\bdy F}$ and $\wh\cL_{\bdy F}$
are isomorphic.

The middle terms $\wh\cR(F)$ and $\wh\cL(F)$ of \eqref{exact sequence 1} and
\eqref{exact sequence 2} are also isomorphic, since both are 
isomorphic to 
\[\overline{R(\cH_F)} \otimes_\R \wh\cA(F)\cong IH^\udot(\M_{\cH_F})\otimes_\R \wh\cA(F).\]
Thus these modules are both generated in 
degrees less than $(1/2)\codim_\R S_F = 2\rk F$, 
and so $\wh\cL_{\bdy F} \cong \wh\cR_{\bdy F}$ 
is also generated in degrees $< 2\rk F$.

By Remark \ref{costalk remark}, $\Ker_{\wh\cL}$ is
isomorphic to $H^\udot_T(S;j^!_S\IC)$, and so it
vanishes in degrees $\le 2\rk F$.  
It follows that $\wh\cL(F)$ and $\wh\cL_{\bdy F}$
have the same dimension in these degrees, and
hence so do $\wh\cR(F)$ and $\wh\cR_{\bdy F}$.  
Our result will follow if we can show that $\Ker_{\wh\cR}$ 
also vanishes in these degrees, since then
the right map of \eqref{exact sequence 1}
must hit all the generators of $\wh\cR_{\bdy F}$.

To see this vanishing, note that by
\eqref{stalk expressions} the 
quotient of $\wh\cR(F, E)$ by the 
generators $e_i$ for $i \in F \setminus E$
is $R(\cH_E) \otimes_\R S(\cH^F)$.  It follows
that $\Ker_{\wh\cR}$ is contained in the 
kernel of the map
\[\wh\cR(F) \cong R(\cH_F) \otimes_\R S(\cH^F) \to \bigoplus_{E < F} R(\cH_E) \otimes_\R S(\cH^F).\]
But this is just $\cR(F,\bdy F) \otimes_\R S(\cH^F)$, 
so the required vanishing follows from Proposition \ref{Combinatorially semi-small}
and the first sentence of the proof of Corollary \ref{rigid MES}. 
\end{proof}

Thus we get a canonical isomorphism $\wh\cR \cong \wh\cL = \Loch(\IC)$, so
we have induced a ring structure on $\wh\cL$, and hence $\IC$ becomes
a ring object in $D^b_T(\M)$.  The uniqueness of this 
ring structure follows from the following result.

\begin{proposition}
The commutative $\wh\cA$-algebra 
structure on $\wh\cR\cong\wh\cL$ is unique.
\end{proposition}

\begin{proof}
Let $m\colon \wh\cR \otimes_{\wh\cA} \wh\cR \to \wh\cR$ be the algebra
structure on $\wh\cR$ that we have already defined, 
and let $m'$ be another one.  We show by 
induction on the rank of a flat $F$ that $m = m'$ on $F$ and $(E, F)$
for any $E \ge F$.  It is obvious that $m = m'$ for the minimal flat 
$\empty$, since $\wh\cR_{\empty} = \wh\cA_{\empty}$.
 
If we know that $m = m'$ at a flat $F$, then the same is true 
for all $(E, F)$ with $E \ge F$, since the map
$\wh\cR(F) \otimes_{\wh\cA(F)} \wh\cA(E,F) \to \wh\cR(E,F)$
is a surjective ring homomorphism for either ring structure.
So we can take $F \ne \emptyset$, and suppose inductively that $m = m'$
on $(F, F')$ for all $F' \le F$. 

The degree restrictions on intersection cohomology stalks
imply that $\wh\cR(F)$ is generated as an $\wh\cA(F)$-module 
in even degrees less than $2\rank F$.  
Similarly, we observed in Remark \ref{costalk remark}
that the kernel of $\wh\cR(F) \to \bigoplus_{F' \le F} \wh\cR(F,F')$
is isomorphic to $H^\udot_T(S;j^!_{S_F}\IC)$, so by the degree restrictions
on intersection cohomology costalks, it is generated as an $\wh\cA(F)$-module in even 
degrees greater than $2\rank F$.  It follows that if $\wh\cR(F)$ has an 
algebra structure which is compatible with the restriction maps, 
then the multiplication by elements in degree $1$ is determined
by the the multiplication on $\wh\cR(F,F')$ for $F' \le F$.  But
$\wh\cR(F)$ with the ring structure $m$ is generated in degree $1$,
and so by associativity we must have $m = m'$ on all elements.
\end{proof}
  
\subsection{Comparing the two localizations}
To complete the proof of Theorem \ref{IC is ring object}, we need to 
compare the two localization functors we have defined.  
As before, we are considering a 
hypertoric variety $\M = \M_\cH$ defined by a central
unimodular arrangement.

For arbitrary
objects $B, C \in D^b_{T,\sS}(\M)$ there is a natural
map $\Loc(B) \otimes_\cA \Loc(C) \to \Loc(B \otimes C)$, 
so our ring structure on $\IC$ induces a ring structure
on the minimal extension sheaf $\cL := \Loc(\IC)$.  
The problem is to show that this agrees with the 
ring structure provided by the isomorphism $\cL \cong \cR$.

Let $\wh\cL = \Loch(\IC)$, with the ring structure
induced from our ring structure on $\IC$ (equivalently,
it is induced from the isomorphism $\wh\cL \cong \wh\cR)$.  
Recall that we have a map
$p\colon \wh{L} \to L$ given by $p((E,F)) = p(F) = F$.
We define a map $\phi\colon p_*\wh\cL \to \cL$ as follows.  Let
$U \subset L$ be an open set, and let $\M_U = \bigcup_{F \in U} S_F$
be the corresponding open subset of $\M$.  The space of 
sections $p_*\wh\cL(U)=\wh\cL(p^{-1}(U))$ is canonically isomorphic to 
$\IH^\udot_T(\M_U)$ by Theorem \ref{localized sheaves}.
This group then maps to $\cL(U)$ by the natural
map $\Gamma_{\IC(\M_U)}$ introduced in Section \ref{localization functor}.
These are all maps of $A = H^\udot_T(pt)$-modules, and they
are compatible with restrictions of open sets, so this
defines a map of $\cA$-modules.  It is also easy to
check that it is a map of rings.

If $U = U_F$ is the smallest open set containing $F$, then
the map $\Gamma_{\IC(\M_U)}$ is just the restriction
$\IH^\udot_T(\M_U) \to H^\udot_T(Tx; \IC)$ to any orbit $Tx \subset S_F$,
followed by the identification of $H^\udot_T(Tx; \IC)$
with $\cL(F) = \cL(U_F)$.  Using Lemma \ref{parity vanishing}(a),
we see that if $s \in \wh\cL(U)$, then the section
$\phi(s)$ sends $F$ to the image of $s_F \in \wh\cL(F)$
in 
\[\wh\cL(F) \otimes_{\wh\cA(F)} \cA(F) = H^\udot_T(S_F; \IC) \otimes_{H^\udot_T(S_F)} H^\udot_T(Tx) \cong H^\udot_T(Tx;\IC) \cong \cL(F).\]
In other words, we can factor $\phi$ as the composition
\[p_*\wh\cL \to p_*(\wh\cL \otimes_{\wh\cA} p^{-1}(\cA)) \to \cL,\]
where the second map is an isomorphism.

Applying the isomorphisms $\wh\cR\cong \wh\cL$, $\cR \cong \cL$,
this second map becomes the map \eqref{relating R and Rhat},
since maps of minimal extension sheaves are rigid.
Thus the ring structure on $\cL$ induced by 
the one on $\IC$ is the same as the one induced
by \eqref{relating R and Rhat} and the isomorphism 
$\cR \cong \cL$, completing the
proof of Theorem \ref{IC is ring object}.

\end{document}